\documentclass[11pt,a4paper]{amsart}
\usepackage[utf8]{inputenc}
\usepackage{amsmath}
\usepackage{amsfonts}
\usepackage{amssymb}
\usepackage[dvipsnames]{xcolor}
\usepackage[smalltableaux]{ytableau}
\usepackage{hyperref}
\usepackage[english]{babel}
\usepackage{mathtools}
\usepackage{geometry}
\usepackage{xcolor}
\usepackage{float}
\usepackage{diagbox}

\usepackage{color}

\allowdisplaybreaks[2] 

\definecolor{lightgray}{gray}{0.7}
\definecolor{extralightgray}{gray}{0.9}
\newcommand{\CoTa}{lightgray}  
\newcommand{\CoTb}{extralightgray}  

\theoremstyle{plain}
\newtheorem{theorem}{Theorem}
\newtheorem{proposition}{Proposition}
\newtheorem{lemma}{Lemma}
\newtheorem{corollary}{Corollary}

\newcommand{\Q}{\mathbb{Q}}
\newcommand{\R}{\mathbb{R}}

\newcommand{\epsi}{\varepsilon}

\newcommand{\intpart}[1]{\left\lfloor#1\right\rfloor}

\DeclareMathOperator{\Res}{Res}
\newcommand{\res}{\text{res}}

\newenvironment{List}{\begin{list}{$\bullet$}{
\setlength{\labelwidth}{.5cm}
\setlength{\leftmargin}{.7cm}
}
}{\end{list}}

\author[F.~Battistoni]{Francesco Battistoni}
\address{Dipartimento di Matematica per le Scienze Economiche, Finanziarie ed Attuariali\\
         Universit\`{a} Cattolica\\
         via Necchi 9\\
         20123 Milano\\
         Italy}
\email{francesco.battistoni@unicatt.it}

\author[G.~Molteni]{Giuseppe Molteni}
\address{Dipartimento di Matematica\\
         Universit\`{a} di Milano\\
         via Saldini 50\\
         20133 Milano\\
         Italy}
\email{giuseppe.molteni1@unimi.it}



\title{Generalized Pohst inequality and small regulators}

\begin{document}

\begin{abstract}
Current methods for the classification of number fields with small regulator depend mainly on an upper
bound for the discriminant, which can be improved by looking for the best possible upper bound of a
specific polynomial function over a hypercube. In this paper, we provide new and effective upper
bounds for the case of fields with one complex embedding and degree between five and nine: this is done
by adapting the strategy we have adopted to study the totally real case, but for this new setting
several new computational issues had to be overcome. As a consequence, we detect the four number fields
of signature $(r_1,r_2)=(6,1)$ with smallest regulator; we also expand current lists of number fields
with small regulator in signatures $(3,1)$, $(4,1)$ and $(5,1)$.
\end{abstract}

\maketitle

\begin{center}
To appear in Math. Comp. 2024 \\
DOI: \url{https://doi.org/10.1090/mcom/3954}
\end{center}

\section{Introduction}
A problem in Algebraic Number Theory consists in giving a complete classification, up to isomorphism, of
number fields with prescribed properties: in particular, scholars are interested in tabulating complete
lists of number fields $K$ such that some chosen invariant is less than some given upper bound. Typical
examples are the classification of fields with small discriminant~\cite{battistoni2019minimum,
battistoni2020small, pohstDegree7TotallyReal, pohst1982computation, pohst1990minimum} and the
classification of imaginary quadratic fields with bounded class number~\cite{HeegnerDiophantische,
starkHeegner, WatkinsClassImaginary}. Both these investigations benefit from the crucial fact that one
can prove that the desired fields occur in finite number, so that it is theoretically possible to
describe them all in a finite time; furthermore, this permits the construction of databases gathering
fields with small discriminant, like the LMFDB database~\cite{lmfdb}, Kl\"{u}ners-Malle's database for fields
with prescribed Galois group~\cite{klunersMalle}, PARIgp tables~\cite{MegrezTables} and Jones-Roberts'
database~\cite{jonesRoberts}.

In this paper we are interested in the classification of number fields $K$ of given signature
$(r_1,r_2)$, i.e. with $r_1$ real embeddings and $r_2$ couples of conjugated complex embeddings, and with
small regulator $R_K$. For every $C>0$, there exists only a finite number of such fields satisfying $R_K
\leq C$ and not of CM type. This is possible thanks to an inequality by Remak~\cite{remak} which bounds
the absolute value of the discriminant $d_K$ in terms of a constant $R(r_1,r_2,C)$ depending on $r_1$,
$r_2$ and $C$ for the not CM case. However, the resulting upper bound is too big to also permit a real
description of all the fields with discriminant within the output range. A procedure to overcome this
difficulty was developed by Astudillo, Diaz y Diaz and Friedman~\cite{regulators}, who combined Remak's
inequality with a lower bound for $R_K$ deduced from explicit formulae of Dedekind Zeta
functions~\cite{friedmanAnalyticRegulator}: the resulting method allowed the aforementioned authors to
obtain the full lists of fields with smallest regulators for the following cases:
\begin{List}
    \item fields of degree at most $6$, in any signature $(r_1,r_2)$,
    \item fields of degree $7$ in any signature apart from $(r_1,r_2)=(5,1)$,
    \item fields of degree $8$ with signature $(r_1,r_2)=(8,0)$ and $(0,4)$,
    \item fields of degree $9$ with signature $(r_1,r_2)=(9,0)$.
\end{List}
The case of the signature $(5,1)$ was solved later by Friedman and Ramirez-Raposo~\cite{RamirezRaposo}
with an ad hoc argument taking advantage of the fact that for this signature there is one non-conjugated
complex embedding. This, together with the success obtained for the totally real signatures $(8,0)$ and
$(9,0)$, suggested that the signature of the considered fields should play a crucial role not only in the
formulation of the problem, but also in the possible success of the specific procedure.

In greater detail, the key ingredient in these procedures is an upper bound as sharp as possible for the
function
\[
  P_{n,r_2}(y_1,\ldots, y_n)\coloneqq \prod_{1\leq i < j\leq n}\left|1-\frac{y_i}{y_j}\right|
\]
where $n\geq 2$ and the $n$-ple $(y_1,\ldots,y_n)$ satisfies the following two conditions:
\begin{List}
\item[A)] The numbers $y_i$ are ordered as $0<|y_1|\leq |y_2|\leq \cdots \leq |y_n|$,
\item[B)] Among the numbers $y_i$, there are $r_2$ couples of complex conjugated and non-real
    numbers, while all the remaining numbers are real.
\end{List}
A simple estimate for $P_{n,r_2}$ is $n^{n/2}$~\cite{bertin} and is deduced by looking at the function as
the determinant of an $n\times n$ matrix with entries bounded by $1$ in absolute value, and applying then
Hadamard's inequality; this upper bound is suitable for the procedure to succeed when the degree is
small, but it is not strong enough as the degree of the considered fields increases; for example this
method applied to the signature $(5,1)$ with the upper bound $7^{7/2}$ is not able to detect the fields
with small regulators. The same phenomenon occurs if one applies the method to the signatures $(6,1),
(4,2)$ and $(2,3)$ in degree $8$ with the upper bound $8^{8/2}$. Friedman and
Ramirez-Raposo~\cite{RamirezRaposo} succeeded to lower the upper bound of $P_{7,1}$ from $7^{7/2} =
907.49\ldots$ to $\exp(6)=403.42\ldots$ and this improvement was enough for the procedure to actually
classify fields with small regulator in this signature.\\
The results for the totally real cases in degree $8$ and $9$ were obtained earlier because the desired
upper bound is consistently smaller whenever only real numbers are considered: in fact,
Pohst~\cite{pohstRegulator} proved that the correct upper bound for the totally real case
is $2^{\intpart{n/2}}$ when $n\leq 11$ already in '77.\\
Recently we were able to extend this bound to every $n$, see~\cite{battistonimolteni} (for a different
proof see also~\cite{Ramirez-Raposo}). This is achieved by developing in greater detail the idea
(initially sketched by Pohst) of associating to the polynomial $P_{n,0}$ a graphical scheme, i.e. a
triangular array such as
\begin{equation}\label{graphScheme1}
\ytableausetup{nosmalltableaux}
\begin{ytableau}
 {+}   & {-}   & {-}   & {+} \\
 \none & {-}   & {-}   & {+} \\
 \none & \none & {+}   & {-} \\
 \none & \none & \none & {-} \\
\end{ytableau}
\end{equation}
where every square represents a factor of $P_{n,0}$ after the change of variables $x_i\coloneqq
y_i/y_{i+1}$ and the sign at place $(i,j)$ represents the sign of the product $\prod_{k=i}^j x_k$, and
then to recognize patterns of signs in any possible graphical scheme that can be replaced with other
patterns producing an upper bound for the corresponding functions. Eventually, we were able to prove that
it is possible to reduce each graphical scheme into the one defined by signs ``$-$'' on the main
diagonal. Finally, this special scheme can be covered by elementary block of signs which are bounded by
known constants. In the example below, the scheme turns out to be estimated by $4$ because each
triangular pattern is estimated by $2$ and the square one is bounded by $1$.
\begin{equation}\label{graphScheme2}
\ytableausetup{nosmalltableaux}
\begin{ytableau}
 {-}    & {+}    & *(\CoTa){-}  & *(\CoTa){+}  \\
 \none  & {-}    & *(\CoTa){+}  & *(\CoTa){-}  \\
 \none  & \none  & *(\CoTb){-}  & *(\CoTb){+}  \\
 \none  & \none  & \none        & *(\CoTb){-}  \\
 \none  & \none  & \none        & \none \\
\end{ytableau}
\quad
\leq 2 \cdot 1\cdot 2 = 4
\end{equation}
\indent %
Motivated by this result, we decided to adapt the technique of graphical schemes to the study of the
functions $P_{n,1}$ and so to the research of fields with one complex embedding and small regulator. The
new setting emerged to be much more intricate: as we will explain in the next sections, more than 150
estimates were needed for the estimate of the new graphical schemes (contrary to the only nine
inequalities needed for the totally real case), and the proof of several of these estimates required
deeper analytic and computational work. The results we obtained are the following.
\begin{theorem}\label{theoremDegree5}
Let $M\coloneqq 3^{15/2}/(4\cdot 7^{7/2})$. Then
\[
  P_{5,1} \leq 16\, M = 16.6965\ldots
\]
and this upper bound is the best possible one.
\end{theorem}
\begin{theorem}\label{theoremDegreesHigher}
We have
\[
  P_{6,1} \leq 34.89,
  \quad
  P_{7,1} \leq 65.81,
  \quad
  P_{8,1} \leq 80,
  \quad
  P_{9,1} \leq 233.1.
\]
\end{theorem}
\noindent %
The results of Theorems~\ref{theoremDegree5} and~\ref{theoremDegreesHigher} constitute meaningful
improvements with respect to all their previous bounds. In particular, the estimate for $P_{8,1}$ allows
to detect the fields of signature $(6,1)$ with smallest regulator.
\begin{corollary}\label{corollaryMinReg}
The first four fields with signature $(r_1,r_2)=(6,1)$ ordered by increasing value of their regulator are
the fields
\[
\begin{array}{cl}
K_1: & x^8 - 2x^7 -  x^6 +   x^5 -  2x^4 +  4x^3 +  3x^2 - 2x - 1 \\[2pt]
K_2: & x^8 - 2x^7 - 3x^6 + 10x^5 -  2x^4 - 11x^3 +  5x^2 + 2x - 1 \\[2pt]
K_3: & x^8        - 9x^6 -  8x^5 + 11x^4 + 21x^3 + 17x^2 + 7x + 1 \\[2pt]
K_4: & x^8        -  x^6 -  3x^5 -  3x^4 +  6x^3 +  4x^2 - 2x - 1
\end{array}
\]
\noindent %
with $R_{K_1}=7.135\ldots$, $R_{K_2} = 7.380\ldots$,  $R_{K_3} = 7.414\ldots$ and $R_{K_4} = 7.430\ldots$.
\end{corollary}
\begin{corollary}\label{corollaryMinRegNonPrim}
The non-primitive field with signature $(r_1,r_2)=(6,1)$ and smallest regulator is the field $K_3$ of the
list in Corollary~\ref{corollaryMinReg}.
\end{corollary}
The new bound for $P_{5,1}$, $P_{6,1}$ and $P_{7,1}$ allows to improve the known lists of fields ordered
according to their regulator in the corresponding signatures. The new lists are described in the next
corollaries.
\begin{corollary}\label{corollaryDeg5}
There exist $40$ fields $K$ with signature $(r_1,r_2){=}(3,1)$ satisfying $R_K\leq 2.15$ and they all
have $|d_K| \leq 25679$.
\end{corollary}
\begin{corollary}\label{corollaryDeg6}
There exist $136$ fields $K$ with signature $(r_1,r_2){=}(4,1)$ satisfying $R_K\leq 4.60$ and they
all have $|d_K|\leq 712603$.
\end{corollary}
\begin{corollary}\label{corollaryDeg7}
There exist $59$ fields $K$ with signature $(r_1,r_2)=(5,1)$ satisfying $R_K\leq 6.10$ and they all
have $|d_K|\leq 7495927$.
\end{corollary}
\indent %
The first author conjectured in~\cite{battistoniCOnjectural} the maximum for the functions studied in
Theorems~\ref{theoremDegree5} and~\ref{theoremDegreesHigher}. In particular, Conjecture~3 appearing there
states that it should satisfy an iterative behaviour as the degree $n$ increases, similarly to what
happens for the totally real case. For the $r_2 = 1$ case the conjecture states that
\[
  P_{n,1} \leq \begin{cases}
                 2^{\frac{n+3}{2}}\, M& n \text{ odd }\geq 5, \\
                 2^{\frac{n+4}{2}}    & n\text{ even }\geq 6.
               \end{cases}
\]
Theorem~\ref{theoremDegree5} proves the conjecture for the case $P_{5,1}$ and the values for $P_{6,1}$
and $P_{8,1}$ in Theorem~\ref{theoremDegreesHigher} are quite close to the conjectural ones.
\smallskip

Here we give a brief description of the next sections of the paper. %
In Section~\ref{sectionProofCoro} we derive the corollaries from the results of
Theorems~\ref{theoremDegree5} and~\ref{theoremDegreesHigher}, and in doing so we recall the steps for the
classification method of number fields with small regulator. %
Section~\ref{SectionTotReal} recalls the framework in which the result for $P_{n,0}$ was obtained,
including the notion of graphical schemes and some basic estimates. %
Section~\ref{sectionChangeVar} introduces the change of variables which allows to study the problem for
the functions $P_{n,1}$ and the idea of ordering of these functions. %
Section~\ref{sectionProofTheorem} presents the proof of Theorem~\ref{theoremDegree5} and the strategy
based on iterated resultants that we used for it. %
Section~\ref{sectionSchemesForTheorem2} illustrates how graphical schemes are adapted in the new context
of fields with one complex embedding and the several results and difficulties that are encountered whenever
one tries to estimate them. %
Section~\ref{sectionTechnicalRemarks} lists some technical and computational remarks regarding the
procedures we used for the estimate of these graphical schemes and the dataset associated. %
Finally, Section~\ref{sectionFinalRemarks} presents some considerations about the upper bounds we found
and their possible improvement, including a discussion about the classification of number fields with
signature $(7,1)$ and small regulator.
\medskip\\
{\bf Acknowledgments:} %
We thank the anonymous referee for their careful reading. We also thank Francesco Fichera and Luca
Galizzi for their constant assistance with the software packages we have used for this work. This project
has been accomplished during the period the first author spent at the Universit\`{a} di Milano; his research
was founded with a grant from the Italian National Research Projects (PRIN) 2017. The first author is
member of GNAMPA research group, the second author of GNSAGA research group.

\section{Proof of the corollaries}\label{sectionProofCoro}
\subsection{Proof of Corollaries~\ref{corollaryMinReg} and~\ref{corollaryMinRegNonPrim}}
We refer to Sections~2 and~3 of~\cite{regulators} for the proof of the propositions of this section.
\begin{proposition}\label{prop1}
Let $K=\Q(\epsi)$ be a field of degree $n$ and signature $(r_1,r_2)$ with $\epsi\in\mathcal{O}_K^*$. Let
$m_K(\epsi)$ be the length of $\epsi$ in the logarithmic lattice of the units of $K$. Let
$\epsi_1,\ldots,\epsi_n$ be the conjugates of $\epsi$ ordered in increasing absolute value. Then
\[
 \log |d_K| \leq 2\log (P_{n,r_2}(\epsi_1,\ldots,\epsi_n)) + m_K(\epsi)\cdot \sqrt{\frac{n^3-n-4r_2^3-2r_2}{3}} =: U_0.
\]
\end{proposition}
\begin{proposition}\label{prop2}
Assume $K=\Q(\epsi)$ as above and that $\epsi$ is such that $m_K(\epsi)$ is the minimum non-zero length
in the logarithmic lattice. Let $r\coloneqq r_1 + r_2 -1$. Then
\[
 m_K(\epsi) \leq \left(\sqrt{r+1} R_K \gamma_r^{r/2}\right)^{1/r}
\]
where $\gamma_j$ is the Hermite constant of dimension $j$.
\end{proposition}
\noindent %
Let us consider now the family of fields $K$ of degree $8$ and signature $(6,1)$ with $R_K \leq
R_0\coloneqq 7.431$. The inequalities of Propositions~\ref{prop1} and~\ref{prop2}, together with the upper
bound for $P_{8,1}$ from Theorem~\ref{theoremDegreesHigher} and the fact that $\gamma_6 = \sqrt[6]{64/3}$
\cite[p.~332]{casselsGeometryNumbers}, imply that, whenever $K$ is generated by a unit with minimum
logarithmic length, we have
\begin{equation}\label{EstimatePrimitive}
 \log |d_{K}|
 \leq 2\log (80) + \Big(\sqrt{7}\cdot 7.431\cdot \sqrt{64/3}\Big)^{1/6}\cdot \sqrt{\frac{8^3-8-4-2}{3}}
 =    36.079500\ldots
\end{equation}
In particular, the hypotheses of Proposition~\ref{prop2} are satisfied whenever $K$ is primitive, i.e.
has no subfields different from $\Q$ and itself. Thus, the first conclusion we can get is that any field
of signature $(6,1)$ with $|d_K| >\exp(36.079500\ldots)$ must have $R_K > 7.431$.\\
The role of the upper bound~\eqref{EstimatePrimitive} is crucial in the next proposition.
\begin{proposition}\label{prop3}
Let $g_{r_1,r_2}:(0,+\infty)\to \R$ be the analytic function defined as
\[
 g_{r_1,r_2}(x)\coloneqq \frac{1}{2^{r_1}4\pi i}\int_{2-i\infty}^{2+i\infty}(\pi^n4^{r_2}x)^{-s/2}(2s-1)\Gamma(s/2)^{r_1}\Gamma(s)^{r_2}ds.
\]
Let $K$ be a field with signature $(r_1,r_2)$ and let $d_1,d_2$ be such that $0<d_1 <
|d_K| < d_2$. If $2g_{r_1,r_2}(1/d_1) > R_0$ and $2g_{r_1,r_2}(1/d_2) > R_0$, then
\begin{equation}\label{lowerBoundRegulator}
 R_K \geq 2g_{r_1,r_2}(1/|d_K|) > R_0.
\end{equation}
If $d_2 \leq \delta$, where $\delta^3$ is a lower bound for the discriminant of a field of signature
$(3r_1,3r_2)$, then all the factors $2$ in the three inequalities involving $R_0$ can be replaced with
$4$.
\end{proposition}
We notice that $\exp(36.079500\ldots)$ is bigger than
$\delta \coloneqq (19.6)^{8}$, with $\delta^3$ a lower bound for the fields of degree 24 and signature $(18,3)$ (this value can
be recovered from Diaz y Diaz' tables of discriminant lower bounds~\cite{y1980tables}). One then verifies
that
\[
 2g_{6,1}(\exp(-36.079500\ldots)) = 11.54216\ldots > 7.431
\]
and
\[
 2g_{6,1}(1/19.6^8) = 37.50073\ldots > 7.431,
\]
so that from Proposition \ref{prop3} the desired fields satisfy $|d_K|\leq (19.6)^8$, which is around $2.1\cdot 10^{10}$.\\
From~\cite{battistoni2020small} we have the list of all the fields $K$ of signature $(6,1)$ with $|d_K|
\leq 79259702$, which is considerably smaller. However, one verifies that
\[
 4g_{6,1}(1/79259702) = 7.48749> 7.431.
\]
Hence from the second part of Proposition~\ref{prop3} it follows that every field $K$ of signature $(6,1)$ with $79259702 <
|d_K| \leq (19.6)^8$ must have $R_K > 7.431$, and so the only possible fields with regulator
below this bound must be contained in the known list (which is formed of eight fields). For every such
field we compute their regulator using PARI/GP: the computation is a priori conditional on GRH and the
output is an integer multiple of the true value, but we can prove that every output is the correct value
for $R_K$ because they satisfy the following condition.
\begin{proposition}\label{prop4}
Let $K$ be a field of signature $(r_1,r_2)$, and $g_{r_1,r_2}$ as in Proposition \ref{prop3}. Let
$\Tilde{R}_K \coloneqq mR_K$ be an integer multiple of the regulator $R_K$. Let $w_K$ be the number of
roots of unity of $K$ (which is two for every field with $r_2 \geq 1$). If
\[
 0 < \frac{\Tilde{R}_K/w_K}{2g_{r_1,r_2}(1/|d_K|)} < 2
\]
then $\Tilde{R}_K = R_K$.
\end{proposition}
The unique fields in the list satisfying $R_K \leq 7.431$ are the fields $K_1$, $K_2$, $K_3$ and $K_4$
described in the statement of Corollary~\ref{corollaryMinReg}. This proves that $K_1$, $K_2$ and $K_4$
are the primitive fields of signature (6,1) with smallest regulator, equal to $7.135\ldots$, $7.38\ldots$
and $7.430\ldots$, respectively. The discussion, however, is not concluded: we have to deal with the case
of non-primitive fields, for which Inequality~\eqref{EstimatePrimitive} is not guaranteed to be correct,
since it may happen that the unit $\epsi$ with smallest logarithmic length generates a proper subfield of
$K$. To deal with the possibility $\Q(\epsi) \subsetneq K$ we need a relative version of
Proposition~\ref{prop1} and some lower bounds on the logarithmic lengths in a tower of fields.
\begin{proposition}\label{prop5}
Let $K,F$ be number fields with $K=F(\epsi)$ where $\epsi\in\mathcal{O}_K^*$. Then
\[
 \log |d_K| \leq [K:F]\log |d_F| + [K:\Q]\log ([K:F]) + m_K(\epsi) C(K/F)
\]
with
\[
 C(K/F)\coloneqq \sqrt{\frac{1}{3}\sum_{v\in\infty_F}([K:F]^3-[K:F]-4r_2(v)^3-2r_2(v))}
\]
where $r_2(v)=0$ unless $v$ is real, in which case $r_2(v)$ is the number of complex places of $K$ lying
above $v$.
\end{proposition}
\begin{proposition}\label{prop6}
If $F \subset K$ and $\eta\in \mathcal{O}_F^*$, then $m_K(\eta) \geq \sqrt{[K:F]} \cdot m_F(\eta)$.
Moreover, if $F=\Q$ and $K$ is totally real, then for $\epsi\in \mathcal{O}_K^*$ we have $m_K(\epsi) \geq
\sqrt{[K:\Q]}\log ((1+\sqrt{5})/2).$
\end{proposition}
Consider a field $K$ of signature (6,1) and let $\eta_1,\eta_2,\ldots,\eta_6$ be its independent units
realizing the successive minima for $m_K$. We have three possible cases to consider, which are the only
possible ones due to the signature of the involved fields:
\begin{List}
\item[a)] $F=\Q(\eta_1)$ is quartic of signature $(4,0)$, and $K=F(\eta_j)$ with $j\in\{2,3,4\}$;

\item[b)] $F=\Q(\eta_1)$ is quadratic totally real, and $K=F(\eta_2)$;

\item[c)] $F=\Q(\eta_1)$ is quadratic totally real, $L=F(\eta_2)$ is quartic of signature $(4,0)$ and
    $K=L(\eta_j)$ with $j\in\{3,4\}$.
\end{List}
For every such possible case we employ Propositions~\ref{prop5} and~\ref{prop6} to obtain an estimate of
the form $\log |d_K| \leq A_0 + \sum_{i=1}^s A_i m_K(\eta_i)$. We have then
\begin{align}
 &\text{a) }\log|d_K| \leq 24\log 2 +2\sqrt{10}\, m_K(\eta_1) + \sqrt{6} \, m_K(\eta_j),                         \quad\qquad\qquad\qquad j\in\{2,3,4\}, \nonumber\\
 &\text{b) }\log|d_K| \leq 24\log 2 +2\sqrt{2} \, m_K(\eta_1) + \sqrt{38}\, m_K(\eta_2),                                                                \nonumber\\
 &\text{c) }\log|d_K| \leq 24\log 2 +2\sqrt{2} \, m_K(\eta_1) + 2\sqrt{2}\, m_K(\eta_2)+\sqrt{6}\, m_K(\eta_j),  \quad                   j\in\{3,4  \}. \label{inequalitiesNonPrim}
\end{align}
This upper bound can then be properly optimized thanks to Lemma~3 of~\cite{regulators} (with the
parameter $\delta$ appearing there set to $\sqrt{8}\log((1+\sqrt{5})/2)$ which is a possible choice
according to Proposition~\ref{prop6}), giving
\[
\log |d_K|
\leq \begin{cases}
        35.237296\ldots & \text{in case a),}\\
        35.701163\ldots & \text{in case b),}\\
        33.821710\ldots & \text{in case c).}
     \end{cases}
\]
Therefore a non-primitive field of signature $(6,1)$ with $R_K {\leq} 7.431$ satisfies $|d_K| \leq
\exp(35.701163\ldots)$: we have $2\cdot g_{6,1}(\exp(-35.701163\ldots)) > 47 > 7.431$ and thus (applying again
the previous propositions) we return to the aforementioned list of eight fields, where we see that the
unique non-primitive field with regulator below 7.431 is $K_3$, which has regulator $7.414\ldots$. This
completes the proof of Corollary~\ref{corollaryMinReg} and Corollary~\ref{corollaryMinRegNonPrim}.

\subsection{Proof of Corollaries~\ref{corollaryDeg5}, \ref{corollaryDeg6} and~\ref{corollaryDeg7}}
The proof of the last three results is similar to the one above in the degree $8$ case. Notice that,
since $5$ and $7$ are primes, the hypothesis of Proposition~\ref{prop1} is satisfied and thus only the
corresponding inequality must be used.

In the degree $5$ case, assume $R_0\coloneqq 2.15$: if $K$ is such that $R_K\leq R_0$, then $P_{5,1}\leq 16\, M$
implies $|d_K|\leq \exp(16.882140\ldots)$. This number is bigger than $(13.136)^5$, which is the
lower bound $\delta$ as in Proposition \ref{prop3} for this signature, hence we have to verify
the condition~\eqref{lowerBoundRegulator} with a factor $2$ instead of $4$. This yields
$2g_{3,1}(\exp(-16.882140\ldots)) = 3.406994\ldots > 2.15$ and $2g_{3,1}(1/(13.136)^5) =
2.158866\ldots > 2.15$. Since $48000 <
(13.136)^5$, we verify that
$4g_{3,1}(1/48000) =
2.157100\ldots > 2.15$ and so the only fields with signature $(3,1)$ and $R_K\leq 2.15$ are among the 145
fields in this signature with $|d_K|\leq 48000$. We download the lists from LMFDB (which are complete up
to this bound) and we put them in PARIgp: a computation of the regulator shows that only 40 fields among
them satisfy $R_K\leq 2.15$, and they all have $|d_K|\leq 25679$. The output regulators are the true
values since the condition described in Proposition~\ref{prop4} is always satisfied.

Now, assume the degree is $6$ and $R_0\coloneqq 4.60$: since $P_{6,1}\leq 34.89$, assuming $R_K\leq R_0$
and mimicking the computations in~\cite[Section~5.2]{regulators} we have that the several relations
between units of minimum length and subfields of $K$ imply
\[
|d_K|\leq\begin{cases}
\exp(24.666347\ldots) & K=\Q(\eta_1),\\
\exp(23.93583\ldots) & \Q(\eta_1)\text{ is a real quadratic subfield of }K,\\
\exp(19.70966\ldots) & \Q(\eta_1)\text{ is a totally real cubic subfield of }K.
\end{cases}
\]
Hence, in any case we have $|d_K|\leq \exp(24.666347\ldots)$ and this upper bound is bigger than
$(15.536)^6$, which is our choice for $\delta$ in this signature. We have
$2g_{4,1}(\exp(-24.666347\ldots)) = 5.083637\ldots > 4.60$ and $2g_{4,1}(1/(15.536)^6) = 5.266992\ldots >
4.60$. Since $1300000 < (15.536)^6$, we further verify that $4g_{4,1}(1/1300000) = 4.631338\ldots > 4.60$:
therefore, the only possible fields with signature $(4,1)$ and $R_K\leq 4.60$ must be among the 613 ones
with $|d_K|\leq 1300000$. A PARIgp computation as above shows that 136 among them have the desired
regulator, and all the output values are the true values because the hypothesis in
Proposition~\ref{prop4} is satisfied.

Finally, assume the degree is $7$ and $R_0\coloneqq 6.10$: then $P_{7,1}\leq 65.81$ and any field with
signature $(5,1)$ and $R_K\leq R_0$ has $|d_K|\leq \exp(30.549708\ldots)$, which is greater than the
number $(17.686)^7$, which is our choice for $\delta$ for this signature. One has
$2g_{5,1}(\exp(-30.549708\ldots))=7.409746\ldots > 6.10$ and $2g_{5,1}(1/(17.686)^7) = 13.744146\ldots >
6.10$. Since $14500000 < (17.686)^7$, we verify that $4g_{5,1}(1/14500000) = 7.060629\ldots > 6.10$, so
that one only has to check if among the 294 fields with signature $(5,1)$ and $|d_K|\leq 14500000$ there
exist some with $R_K\leq 6.10$. The usual PARIgp computation on the complete LMFDB list finds 59 of them,
and they all satisfy $|d_K|\leq 7495927$.

\section{Recalls on the totally real case}\label{SectionTotReal}
In this section we briefly recall how the supremum for $P_{n,0}$ was obtained: this discussion will set
the foundation for the work on the function $P_{n,1}$. All the results in this section are proved
in~\cite{battistonimolteni}.

In the totally real case all the numbers $y_i$ satisfying A) and B) that we consider are real. Setting
the change of variables
\begin{equation}\label{changeOfVariablesReal}
 x_i\coloneqq \frac{y_i}{y_{i+1}},
 \qquad
 i=1,\ldots,n-1
\end{equation}
the function $P_{n,0}$ becomes the polynomial in $n-1$ variables (each one in $[-1,1]$)
\[
Q_{n-1}(x_1,\ldots,x_{n-1}) = \prod_{i=1}^{n-1}\prod_{j=i}^{n-1}\left(1-\prod_{k=i}^j x_k\right)
\]
and looking for the supremum of $P_{n,0}$ is equivalent to looking for the maximum of $Q_{n-1}$ over
$[-1,1]^{n-1}$. Instead of directly studying the function $Q_{n-1}$, one considers all the $2^{n-1}$
subcases defined by a choice for the signs of the variables; if $\mathbf{\epsi}\coloneqq
(\epsi_1,\ldots,\epsi_{n-1})$ is a vector formed by elements $\epsi_i\in\{\pm 1\}$, we consider the
$2^{n-1}$ configurations
\begin{equation}\label{configuration}
Q_{n-1,\mathbf{\epsi}}(x_1,\ldots,x_{n-1}) = \prod_{i=1}^{n-1}\prod_{j=i}^{n-1}\left(1-\prod_{k=i}^j\epsi_k\prod_{k=i}^j x_k\right)
\end{equation}
where now each variable is assumed to be in $[0,1]$.

Using the notation of~\cite{battistonimolteni}, we recall the notion of \emph{graphical scheme}, i.e. a
triangular $n\times n$ array $C$ which has only symbols ``+'' or ``-'' in its entries $C_{i,j}$ for
$1\leq i\leq j\leq n-1$. \eqref{graphScheme1} and~\eqref{graphScheme2} are examples of graphical schemes.
A function $F_C$ defined over $[0,1]^{n-1}$ can be associated to each graphical scheme $C$, and it has
the form:
\[
F_C(x_1,\ldots,x_{n-1}) = \prod_{i=1}^{n-1}\prod_{j=i}^{n-1}\left(1-C_{i,j}\prod_{k=i}^j x_k\right).
\]
The configuration~\eqref{configuration} is thereby the function $F_C$ for the graphical scheme $C$ with entries
\[
C_{i,j} =\begin{cases}
+ & \text{if }\prod_{k=i}^j\epsi_k = 1,\\
- & \text{if }\prod_{k=i}^j\epsi_k = -1.
\end{cases}
\]
Given two graphical schemes $C$ and $C'$, we say that $C \leq C'$ if $F_C\leq F_{C'}$. Estimates of
graphical schemes can be obtained by recognizing patterns in the starting scheme $C$ and replacing them
with a new pattern such that the replacement corresponds to an estimate between the corresponding
factors: the resulting scheme $C'$ will then satisfy $C\leq C'$. We denote these estimates as
\emph{dynamical estimates}. We recall four dynamical estimates.
\begin{lemma}\label{lemmaDynamical}
Let $C$ be a graphical scheme.
\begin{List}
\item[P)] Let $C'$ be the graphical scheme defined changing the element
    $\ytableausetup{smalltableaux,aligntableaux=bottom}
     \begin{ytableau}
      {+}  \\
     \end{ytableau}
    $
    into
    $
     \begin{ytableau}
      {-}  \\
     \end{ytableau}
    $
    and keeping everything else unchanged. Then $C \leq C'$.
\item[H)]
    Let $C'$ be the graphical scheme defined changing the two elements
    $\ytableausetup{smalltableaux,aligntableaux=bottom}
    \begin{ytableau}
     {+} & {-}  \\
    \end{ytableau}
    $,
    not necessarily adjacent, into
    $\ytableausetup{smalltableaux,aligntableaux=bottom}
    \begin{ytableau}
     {-} & {+}  \\
    \end{ytableau}
    $
    and keeping everything else unchanged. Then $C \leq C'$.
\item[V)] Let $C'$ be the graphical scheme defined changing the two elements
    $\ytableausetup{smalltableaux,aligntableaux=bottom}
    \begin{ytableau}
     {-} \\
     {+} \\
    \end{ytableau}
    $,
    not necessarily adjacent, into
    $
    \begin{ytableau}
     {+} \\
     {-} \\
    \end{ytableau}
    $
    and keeping everything else unchanged. Then $C \leq C'$.
\item[S)] Let $C'$ be the graphical scheme defined changing the four elements
    $\ytableausetup{smalltableaux,aligntableaux=bottom}
    \begin{ytableau}
     {-} & {+} \\
     {+} & {-} \\
    \end{ytableau}
    $,
    not necessarily adjacent, into
    $
    \begin{ytableau}
     {+} & {-} \\
     {-} & {+} \\
    \end{ytableau}
    $
    and keeping everything else unchanged. Then $C \leq C'$.
\end{List}
\end{lemma}
These four estimates are sufficient, as the following result states.
\begin{theorem}[Th.2 in~\cite{battistonimolteni}]
Let $C_{n-1,\epsi}$ be the graphical scheme of a configuration $Q_{n-1,\epsi}$. Let $C_{n-1,-}$ be the
graphical scheme of the configuration $Q_{n-1,\epsi_{-}}$ which has vector of signs $\epsi_{-} \coloneqq
(-1,-1,\ldots,-1)$. Then $C_{n-1,\epsi} \leq C_{n-1,-}$: the estimate is obtained by applying only
estimates of the form P, H, V and S to the entries of $C_{n-1,\epsi}$.
\end{theorem}
The desired maximum for our configurations can thus be obtained by just studying the graphical scheme
$C_{n-1,-}$, which is of the form
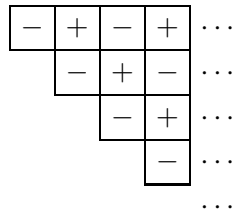
\[
\ytableausetup{nosmalltableaux}
\begin{ytableau}
 {-}   & {+}   & {-}   & {+}   & \none[\ \cdots]\\
 \none & {-}   & {+}   & {-}   & \none[\ \cdots]\\
 \none & \none & {-}   & {+}   & \none[\ \cdots]\\
 \none & \none & \none & {-}   & \none[\ \cdots]\\
 \none & \none & \none & \none & \none[\ \cdots]\\
\end{ytableau}
\]
with every line made of alternating signs. Patterns in the scheme correspond to factors of the
corresponding function, and each such factor may be bounded by suitable constants.
\begin{lemma}\label{lemmaStaticEstimates}
Let $C$ be a graphical scheme. The following patterns, if contained in $C$, can be estimated with the
following constants:
\begin{List}
    \item  $\ytableausetup{smalltableaux,aligntableaux=bottom}
        \begin{ytableau}
         \none     & \none[j]\\
         \none[i]  & {+}     \\
        \end{ytableau}
        \leq 1
        $,
        i.e. $F_{C_{i,j}} \leq 1$.
    \item $\ytableausetup{smalltableaux,aligntableaux=bottom}
        \begin{ytableau}
         \none     & \none[j] & \none[j']\\
         \none[i]  & {+}      & {-}      \\
        \end{ytableau} \leq 1
        $,
        i.e. $F_{C_{i,j}} F_{C_{i,j'}} \leq 1$.
    \item $\ytableausetup{smalltableaux,aligntableaux=bottom}
        \begin{ytableau}
         \none     & \none[j] \\
         \none[i]  & {-}      \\
         \none[i'] & {+}      \\
        \end{ytableau} \leq 1
        $,
        i.e. $F_{C_{i,j}} F_{C_{i',j}} \leq 1$.
    \item $\ytableausetup{smalltableaux,aligntableaux=bottom}
        \begin{ytableau}
         \none     & \none[j] & \none[j']\\
         \none[i]  & {-}      & {+}      \\
         \none[i'] & {+}      & {-}      \\
        \end{ytableau}
        \leq 1$,
        i.e. $F_{C_{i,j}} F_{C_{i',j}} F_{C_{i,j'}} F_{C_{i',j'}} \leq 1$.
    \item Assume $j'=j+1$. Then\\ $\ytableausetup{smalltableaux,aligntableaux=bottom}
        \begin{ytableau}
         \none     & \none[j] & \none[j']\\
         \none[i]  & {-}      & {+}      \\
         \none[i'] &    \none & {-}      \\
        \end{ytableau}
        \ \ ,\
        \begin{ytableau}
         \none     & \none[j] & \none[j']\\
         \none[i]  & {-}      & {-}      \\
         \none[i'] &    \none & {+}      \\
        \end{ytableau}
        \ \ ,\
        \begin{ytableau}
         \none     & \none[j] & \none[j']\\
         \none[i]  & {+}      & {-}      \\
         \none[i'] &    \none & {-}      \\
        \end{ytableau}$
        are all $\leq 1$ i.e. $F_{C_{i,j}} F_{C_{i,j'}} F_{C_{i',j'}} \leq 1$ for each of them.
    \item Assume $i'=i+1$. Then\\ $\ytableausetup{smalltableaux,aligntableaux=bottom}
        \begin{ytableau}
         \none     & \none[j] & \none[j']\\
         \none[i]  & {-}      & {+}      \\
         \none[i'] & \none    & {-}      \\
        \end{ytableau}
        \ \ ,\
        \begin{ytableau}
         \none     & \none[j] & \none[j']\\
         \none[i]  & {-}      & {-}      \\
         \none[i'] & \none    & {+}      \\
        \end{ytableau}
        \ \ ,\
        \begin{ytableau}
         \none     & \none[j] & \none[j']\\
         \none[i]  & {+}      & {-}      \\
         \none[i'] & \none    & {-}      \\
        \end{ytableau}$
        are all $\leq 1$ i.e. $F_{C_{i,j}}\cdot  F_{C_{i,j'}}\cdot F_{C_{i',j'}} \leq 1$ for each of them.
    \end{List}
\end{lemma}
These \emph{static estimates} are what is needed in order to recover the result for $C_{n-1,-}$.
\begin{theorem}[Lemma~3 and Th. 1 in~\cite{battistonimolteni}]
One has $  C_{n-1,-} \leq 2^{\intpart{n/2}}$ so that this is the maximum of $Q_{n-1}$ on $[-1,1]^{n-1}$.
This bound cannot be improved since it is attained at the point $(-1,0,-1,0,\ldots)$, so that the
supremum of $P_{n,0}$ is $2^{\intpart{n/2}}$.
\end{theorem}

\section{The change of variables for \texorpdfstring{$r_2=1$}{r2=1}}\label{sectionChangeVar}
Starting from the previous considerations for the totally real case, we begin the investigation for the
function $P_{n,1}$. As before, we would like to transform it into a polynomial function via a change of
variables and to look for the maximum of this new function over a proper domain. However, this
transformation will no longer be immediate due to the presence of a couple of complex conjugated numbers
among the $y_i$.\\
In fact, we have to consider an $n$-ple $(y_1,y_2,\ldots,y_k,\overline{y_k},y_{k+1},\ldots,y_{n-1})$
which satisfies condition A). The change of variables~\eqref{changeOfVariablesReal} cannot be applied
directly, since quotients $y_{k-1}/y_k$ and $\overline{y_k}/y_{k+1}$ are not real. A more suitable change
of variables is instead the following:
\begin{equation}\label{changeOfVarComplex}
 x_i\coloneqq\begin{cases}
                y_i/y_{i+1},              & i\neq k-1, k \\
                y_{k-1}/|y_k|,            & i = k-1      \\
                |\overline{y_k}|/y_{k+1}, & i = k
             \end{cases},
\quad
 g \coloneqq \cos(\arg y_k).
\end{equation}
Several things have to be remarked about this transformation: first of all, the resulting function will
no longer be purely polynomial, since it will always contain the term $2\sqrt{1-g^2}$, corresponding to
the factor $|1-y_k/\overline{y_k}|$; moreover, and most importantly, different functions will arise as
different indexes $k$ are chosen for the position of the complex conjugated couple
$(y_k,\overline{y_k})$.

We define \emph{$k$-th ordering of $P_{n,1}$} the function $L_{n,k}$ resulting from the change of
variables~\eqref{changeOfVarComplex} applied to $P_{n,1}$ with the complex couple
$(y_{n-k},\overline{y_{n-k}})$. Here are the first and the second ordering for $n=5$, corresponding to
indexes $k=4$ and $k=3$ for the complex couple.
\begin{align*}
   L_{5,1} =& (1-x_1)         (1-x_1 x_2)  (1-2x_1 x_2 x_3 g + (x_1 x_2 x_3)^2)     \\
            &\phantom{(1-x_1)}(1-x_2)\ \ \ (1-2x_2 x_3 g + (x_2 x_3)^2)             \\
            &\phantom{(1-x_1)(1-x_1 x_2)}  (1-2 x_3 g + x_3^2)\cdot 2\sqrt{1-g^2},  \\[.5\baselineskip]
   L_{5,2} =& (1-x_1)         (1-2 x_1 x_2 g + (x_1 x_2)^2)          (1-x_1 x_2 x_3)\\
            &\phantom{(1-x_1)}(1-2 x_2 g + x_2^2)\ \ \ \ \ \ \ \ \ \,(1-x_2 x_3)    \\
            &\phantom{(1-x_1)(1-2 x_1 x_2 g + (x_1 x_2)^2)}          (1-2 x_3 g + x_3^2)\cdot 2\sqrt{1-g^2}.
\end{align*}
Notice that the polynomial factors containing $g$ and different from the square root correspond to either
products of the form $|1-y_i/y_k|\cdot |1-y_i/\overline{y_k}| = |1-y_i/y_k|^2$ or $|1-y_k/y_j|\cdot
|1-\overline{y_k}/y_j| = |1-y_k/y_j|^2$.

A priori we should then study all the $n-1$ possible orderings resulting from this change of variables,
since every time we end up with a different function. Fortunately, there is a symmetry between the
orderings which allows to discard half of the cases.
\begin{lemma}
Let $n\geq 3$ and $1\leq k\leq n-1$. Then
\[
 L_{n,k}(x_1,x_2,\ldots,x_{n-2},g) = L_{n,n-k}(x_{n-2},x_{n-3},\ldots,x_1,g).
\]
\end{lemma}
\begin{proof}
The sequence
\[
0<|y_1|\leq |y_2|\leq\cdots \leq |y_k|=|\overline{y_k}|\leq\cdots\leq |y_{n-1}|
\]
is equivalent to the sequence
\[
0 < \left|\frac{1}{y_{n-1}}\right|
\leq \cdots
\leq \left|\frac{1}{y_k}\right|
\leq \left|\frac{1}{\overline{y_k}}\right|
\leq \cdots
\leq \left|\frac{1}{y_1}\right|.
\]
The claim follows by constructing the function $L_{n,n-k}$ using the numbers $z_j\coloneqq 1/y_{n-j}$.
\end{proof}
\noindent %
{\bf Remark:} the quantity $P_{n,1}$ can be seen to be estimated by the quantity
$2^{\intpart{(5n-8)/2}}$: in fact, if $y_k$ and $\overline{y_k}$ form the complex conjugated couple,
there are $n-2$ factors of the form $|1-y_j/y_k|^2$ or $|1-y_k/y_j|^2$ with $j\neq k$, each term bounded
by $4$. The term $|1-y_j/\overline{y_j}|$ is bounded by $2$ and removing all these terms from $P_{n,1}$
we obtain a function which becomes $P_{n-2,0}$ via an additional change of variable and which is at most
$2^{\intpart{(n-2)/2}}$. The product $P_{n,1}$ is thus estimated by $2\cdot 4^{n-2}\cdot
2^{\intpart{(n-2)/2}} = 2^{\intpart{(5n-8)/2}}$. However, this estimate turns out to be better than the
trivial bound $n^{n/2}$ only for $n\geq 25$, so it is of no use for the applications we study in
Theorems~\ref{theoremDegree5} and~\ref{theoremDegreesHigher}.

\section{Proof of Theorem~\ref{theoremDegree5}}\label{sectionProofTheorem}
In order to prove an estimate for $P_{5,1}$, we consider the only two possible orderings for its
transformation via the change of variables~\eqref{changeOfVarComplex}, which are the examples $L_{5,1}$
and $L_{5,2}$ shown in the previous section. The study of the two functions will be quite similar, but we
consider them separately since some meaningful differences will occur; moreover, the techniques employed
in this section will be later employed for the study of functions $P_{n,1}$ with $n\geq 6$.
\smallskip\\
{\bf Remark:} in the following lines, we shall often factorize polynomials in several variables with
rational coefficients. The factorization results turned to be available thanks to the computer algebra
MAGMA~\cite{cannon2011handbook}. Moreover, real roots of rational polynomials shall be computed: this
have been made by employing the computer algebra PARIgp~\cite{pari}. MAGMA and PARIgp files describing
the details of these computations can be found in~\cite{DATA}.

\subsection{Estimate for \texorpdfstring{$\mathbf{L_{5,1}}$}{L5-1}}
The first ordering was partially studied in~\cite{battistoniCOnjectural} as a toy model for the
formulation of the conjectures about upper bounds for $P_{n,1}$: there, a partial result about the
maximum of $L_{5,1}$ is proved.
\begin{lemma}\label{lemmaBoundary5_1}
The function $L_{5,1}(x_1,x_2,x_3,g)$ assumes its maximum over $[-1,1]^4$ for $x_3=1$ and $g\neq \pm 1$.
\end{lemma}
\noindent
Under this assumption, $L_{5,1}$ reduces to the simpler expression
\begin{align*}
 & (1-x_1)         (1-x_1 x_2)  (1-2x_1 x_2 g + (x_1 x_2 )^2)\\
 &\phantom{(1-x_1)}(1-x_2)\ \ \ (1-2x_2  g + x_2^2)          \\
 &\phantom{(1-x_1)(1-x_1 x_2)}  4\cdot (1-g)\sqrt{1-g^2}.
\end{align*}
The search of the maximum for this function over $[-1,1]^2\times (-1,1)$ is carried through various
steps: first of all, we study the behaviour of the function whenever one of the variables is equal to $0$
(this will turn useful for next computations). We have:
\begin{List}
 \item $x_1 = 0$: the function is $L_{5,1}(0,x_2,1,g) = L_{4,1}(x_2,1,g)$ and we know this function is
     at most 16 (see the remark in Section 1). This value is attained at $x_2=-1$ and $g=0$.
 \item $x_2 = 0$: the function is $4\cdot (1-x_1)(1-g)\sqrt{1-g^2}$ which is maximized at $x_1 = -1$
     and $g = -1/2$ giving $6\sqrt{3}=10.392\ldots$(this is verified by studying the partial
     derivatives of the given function).
 \item $g=0$: the function becomes
     \[
      L = 4(1-x_1)(1-x_1 x_2)(1 + (x_1 x_2 )^2) (1-x_2)(1 + x_2^2).
     \]
     We look for the maximum of $L$ over $[-1,1]^2$. First of all, we determine whether there are
     stationary points in the interior $(-1,1)^2$: we derive $L$ with respect to $x_1$ and denote by
     $L_{x_1}$ the only factor of the derivative (which is a polynomial) which is not zero on the
     boundary. This gives
     \[
      L_{x_1} = 4x_2^3x_1^3 + (-3x_2^3 - 3x_2^2)x_1^2 + (2x_2^2 + 2x_2)x_1 + (-x_1 - 1).
     \]
     We define $L_{x_2}$ in the same way but from the derivative with respect to $x_2$ and we obtain
     \begin{align*}
      L_{x_2} =&   (6x_2^5 - 5x_2^4 + 4x_2^3 - 3x_2^2)x_1^3 + (-5x_2^4 + 4x_2^3 - 3x_2^2 + 2x_2)x_1^2 \\
               & + (4x_2^3 - 3x_2^2 + 2x_2 - 1)x_1 + (-3x_2^2 + 2x_2 - 1).
     \end{align*}
     An eventual stationary point must be a common root $(\alpha,\beta)$ of $L_{x_1}$ and $L_{x_2}$, thus
     its coordinate $\alpha$ must be a root of the resultant between the two polynomials with respect to
     the variable $x_2$. This resultant is equal to
     \[
      32x_1^9 + 24x_1^8 + 20x_1^7 + 7x_1^6 + 24x_1^5 + 67x_1^4 + 15x_1^2 + 27.
     \]
     However, this polynomial has no real roots between $-1$ and $1$, so this means that $L$ has no stationary
     points in the interior.

     One then studies the behaviour of $L$ on the boundaries: clearly $L$ is $0$ when $x_1 = 1$ or $x_2=1$,
     while at $x_1=-1$ we have
     \[
      L = 8(1-x_2^4)(1+x_2^2)  \leq 8\cdot \frac{32}{27} = 9.481\ldots
     \]
     (the upper bound is obtained at $x_2=\pm \frac{1}{\sqrt{3}}$)  and at $x_2=-1$ we have
     \[
      L = 16(1-x_1^4) \leq 16.
     \]
\end{List}
Thus we have proved that, whenever one of the three variables is equal to $0$, the function $L_{5,1}$ is
at most $16$: this information will be useful for the next computations. After this preliminary
discussion, we have two things to consider in order to optimize $L_{5,1}$: we must study the behaviour of
the function on the boundaries of $[-1,1]^3$ and then we must look for possible stationary points in the
interior of this cube. Let us begin with the boundary investigation.
\begin{List}
\item $x_1 = 1$: the function $L_{5,1}$ assumes the value 0 under this condition. The same holds for
    the boundaries $x_2 = 1$ and $g = \pm 1$.
\item $x_1 = -1$: the function $L_{5,1}$ becomes
    \[
    S = 8\cdot(1-x_2^2)((1+x_2^2)^2-4x_2^2g^2)(1-g)\cdot\sqrt{1-g^2}.
    \]
    Again, we would like to consider the partial derivatives of $S$ with respect to $x_2$ and $g$ in
    order to find an eventual maximum point (notice that in this case $S$ becomes 0 whenever a boundary
    condition is satisfied by either $x_2$ or $g$). However, we do not compute directly the derivatives
    of $S$ since we do not want to carry the square root term in the research of the points: we consider
    thus only the term
    \[
    L = 8\cdot(1-x_2^2)((1+x_2^2)^2-4x_2^2g^2)(1-g)
    \]
    and we study the system
    \[
    \begin{cases}
    \frac{\partial L}{\partial x_2} = 0,\\
    \frac{\partial L}{\partial g}(1-g^2) - L g = 0,
    \end{cases}
    \]
    where the second quantity is obtained from the derivative rule for $S$ with respect to $g$. We
    factorize the left hand sides of these equations and we discard any factor that either has roots only
    on the boundary or has roots only when some variable is zero: this is because the necessary
    considerations have been already made or will be done in further boundary study. We are thus left
    with the factors
    \begin{align*}
    L_{x_2} &= 3x_2^4 + (-8g^2 + 2)x_2^2 + (4g^2 - 1)
    \intertext{and}
    L_g &= (2g + 1)x_2^4 + (-16g^3 - 4g^2 + 12g + 2)x_2^2 + (2g + 1).
    \end{align*}
    and we study the system $L_{x_2} = L_g = 0$. A stationary point for $S$ must be then a common root
    $(\beta,\gamma)$ of $L_{x_2}$ and $L_g$, hence $\beta$ must be  a root of their resultant with
    respect to $g$, which is $8x_2^4 + x_2^2 - 1$. This polynomial has two roots between $-1$ and 1, and
    substitution of these roots in $L_{x_2}$  provides four stationary points $(\pm \beta,\pm \gamma)$
    where $\beta = 0.54455\ldots$ and $\gamma = 0.29653$. One then verifies that $S(\pm\beta,\gamma) =
    5.9612\ldots$ and $S(\beta,\pm\gamma) = 10.9870\ldots$ (so we are below the value 16 we have found
    before).
\item $x_2 = -1$: this case has been studied in~\cite[Conjecture 1]{battistoniCOnjectural}, where it was
    proved that under the assumption $x_2 = -1$ and $x_3 = 1$ the function $L_{5,1}$ has a maximum at the
    point $(x_1,x_2,x_3,g) = (1/\sqrt{7},-1,1,1/(2\sqrt{7}))$ and the maximum value is exactly $16M$ (the
    first author stated in~\cite{battistoniCOnjectural} that the determination of this maximum for
    $L_{5,1}$ was still an open problem because of the lack of techniques for a rigorous optimization
    over the interior of $[-1,1]^3$).
\end{List}
Finally, we discuss the possible existence of maximum points for $L_{5,1}$ in the open set $(-1,1)^3$.
First of all, we compute the derivatives
\[
\frac{\partial L}{\partial x_1},
\quad
\frac{\partial L}{\partial x_2},
\quad
\frac{\partial L}{\partial g}(1-g^2) - L g.
\]
We factorize them and then we denote by $L_{x_1}$, $L_{x_2}$ and $L_g$ the unique factors of each
derivative which has no roots in the boundary or for some variable equal to $0$. We study then the system
$L_{x_1}=L_{x_2}=L_g=0$: common zeros $(\alpha,\beta,\gamma)$ of these three objects must give also
common zeros $(\alpha,\beta)$ of the resultants
\[
\Res(L_{x_2},L_{x_1};g),
\quad
\Res(L_g,L_{x_1};g).
\]
We denote by $\res1g$ and $\res2g$ the unique factors of these resultants which again have no roots on
the boundary or for some variable equal to zero, and we study the system $\res1g=\res2g=0$. Finally, we
look for the $x_1$-coordinate $\alpha$ of our stationary points by studying the factors of
$\Res(\res1g,\res2g;x_2)$: one verifies that the only acceptable values for $\alpha$ (i.e. roots of this
resultant which are in $(-1,1)$) are $-1/2$ and $3/5$. But if we substitute $x_1=-1/2$ in $\res1g$, we
notice that the only roots of the new polynomial (i.e. the only acceptable values for $\beta$) are $x_2 =
0$ or $x_2=-1$, which constitute boundary cases which we already studied. If instead we substitute
$x_1=3/5$ in $\res1g$, we find a unique admissible value $\beta = 0.7373\ldots$ but substitution of both
$x_1 = \alpha$ and $x_2 = \beta$ in $L_x$ gives no admissible values for $g$.

This means that we do not find stationary points in $(-1,1)^3$ for the function $L_{5,1}$: gathering all
the results, this proves that $L_{5,1}\leq 16\, M$ over $[-1,1]^4$ and this upper bound is attained at
the point $(x_1,x_2,x_3,g) = (1/\sqrt{7},-1,1,1/(2\sqrt{7}))$.

\subsection{Estimate for \texorpdfstring{$\mathbf{L_{5,2}}$}{L5-2}}
In this case we do not have a result similar to Lemma~\ref{lemmaBoundary5_1} which allows to remove one
variable from the optimization. However, it is straightforward to verify that
\[
L_{5,2}(x_1,x_2,1,g) = L_{5,1}(x_1,x_2,1,g) = L_{5,2}(x_1,-x_2,-1,-g)
\]
and thus anytime we are reduced to a case in our computation where $x_3=\pm 1$, we already know that
$L_{5,2}$ is at most $16\, M$. We optimize this function by following the steps we presented for the
function $L_{5,1}$: so we study the cases where one variable is zero as the first thing.
\begin{List}
\item $x_1=0$: just like for $L_{5,1}$, we obtain a function which is a transformation of $P_{4,1}$ via a
    change of variables~\eqref{changeOfVarComplex}, so we know it is $\leq 16$ (and the value is attained
    at $x_2=-1, x_3=1, g = 0$).
\item $x_2 = 0$: this is the same as the case $x_2=0$ for $L_{5,1}$, so we already know it is at most
    $6\sqrt{3}=10.392\ldots$.
\item $x_3=0$: just like $x_1=0$, this case provides a transformation of $P_{4,1}$, and so the function
    under this condition is at most 16.
\item $g= 0$: the function becomes
    \[
    L = 2\cdot(1-x_1)(1+ (x_1 x_2)^2)(1-x_1 x_2 x_3)(1+x_2^2)(1-x_2 x_3)
            (1 + x_3^2).
    \]
    Again, we first look for stationary points in the interior $(-1,1)^3$ for the variables $(x_1,x_2,x_3)$
    and then we look at the behaviour on the boundaries of the three-dimensional cube.\\
    For the stationary points, we need common roots of the derivatives
    \[
    \frac{\partial L}{\partial x_1},
    \quad
    \frac{\partial L}{\partial x_2},
    \quad
    \frac{\partial L}{\partial x_3}
    \]
    and we denote by $L_{x_1}$, $L_{x_2}$ and $L_{x_3}$ the unique factors of each derivative which is
    not always positive and has no roots only in the boundary or for some variable equal to 0. A common
    zero $(\alpha,\beta,\gamma)$ of these objects must give a common zero $(\alpha,\beta)$ for the
    resultants
    \[
    \Res(L_{x_1},L_{x_2};x_3),
    \quad
    \Res(L_{x_3},L_{x_2};x_3).
    \]
    We factorize the first resultant and we keep only the factor without roots on the boundary or for a
    variable equal to 0: we call it $\text{res1}x_3$. We do the same with the second resultant, with the
    remark that the non-trivial factors are now two (both different from $\text{res1}x_3$): we denote
    their product as $\text{res2}x_3$. Finally, the common zero for these two resultants must give a root
    $\alpha$ of their resultant with respect to the variable $x_2$, which is
    \begin{align*}
    21840x_1^{11} &+ 107944x_1^{10} + 280513x_1^9 + 529240x_1^8 + 728752x_1^7 + 802042x_1^6 \\
    &+ 728398x_1^5 + 511918x_1^4 + 311680x_1^3 + 126574x_1^2 + 44721x_1 + 5418.
    \end{align*}
    This polynomial has only one root $\alpha\in (-1,1)$: substitution of $x_1=\alpha$ in
    $\text{res2}x_3$ gives two possible values for $\beta$, and substitution of $x_1=\alpha$ and
    $x_2=\beta$ in $L_{x_3}$ gives two values for $\gamma$. We end with two stationary points, and the
    function $L$ assumes at both the value $3.0285\ldots$

    We see now what happens to $L$ if we assume boundary conditions on $x_1$ and $x_2$: at $x_1=1$ the
    function is trivially zero, while at $x_1=-1$ it becomes
    \[
     4\cdot(1+ x_2^2)(1+x_2^2)(1-(x_2 x_3)^2)(1 + x_3^2).
    \]
    The first factor $(1+x_2^2)$ is trivially $\leq 2$, while the remaining three factors are estimated
    by $2$ thanks to Lemma~\eqref{lemmaStaticEstimates}. Hence the function $L$ is at most $16$ in this
    case.

    If we assume $x_2 = 1$ the function $L$ becomes
    \[
    4\cdot(1-x_1)(1+ x_1^2)(1-x_1  x_3)(1- x_3)(1 + x_3^2) = 4\cdot F.
    \]
    By using partial derivatives and standard boundary optimization in two variables, the factor $F$
    results to be at most $4$, so that also in this case we obtain $L\leq 16$. Finally, we notice that
    for $x_2=-1$ there is nothing to prove, since $L(x_1,-1,x_3) = L(x_1,1,-x_3)$.
\end{List}
Once this preliminary case for the variables equal to zero is discussed, we begin by studying what
happens in the more general cases of either $x_1$ or $x_2$ being equal to $\pm 1$.
\begin{List}
    \item $x_1=1$: the function is trivially zero.
    \item $x_2 = 1$: the function becomes
    \[
    S = 4\cdot (1-x_1)(1-2 x_1 g + x_1^2)(1-x_1 x_3)(1-g)(1- x_3)(1-2 x_3 g + x_3^2)\cdot \sqrt{1-g^2}.
    \]
    Let $L\coloneqq S/\sqrt{1-g^2}$.
    The procedure of studying successive resultants employed before and applied to the quantities
    \[
     \frac{\partial L}{\partial x_1},
     \quad
     \frac{\partial L}{\partial x_3},
     \quad
     \frac{\partial L}{\partial g} (1-g^2) - L g
    \]
    finds no stationary points in the interior of $[-1,1]^3$. Assuming then $x_1=-1$ (which is the only
    meaningful boundary condition we can impose on this case) we obtain
    \[
    16\cdot (1-x_3^2)(1-2 x_3 g + x_3^2)(1-g^2)\cdot \sqrt{1-g^2}.
    \]
    This is exactly the function with maximum equal to $16\cdot M$ (attained at $x_3=1/\sqrt{7}$ and
    $g=1/(2\sqrt{7})$).
    \item $x_2 = -1$: this case reduces to the previous one since it is immediate to verify that
    \[
    L_{5,2}(x_1,-1,x_3,g) = L_{5,2}(x_1,1,-x_3,-g).
    \]
    \item $x_1 = -1$: the function becomes
    \[
     S =  4\cdot (1+2  x_2 g + x_2^2)(1-2 x_2 g + x_2^2)(1-(x_2 x_3)^2)(1-2 x_3 g + x_3^2)\cdot\sqrt{1-g^2}.
    \]
    Let $L\coloneqq S/\sqrt{1-g^2}$.
    The procedure of studying successive resultants applied to the quantities
    \[
     \frac{\partial L}{\partial x_2},
     \quad
     \frac{\partial L}{\partial x_3},
     \quad
     \frac{\partial L}{\partial g} (1-g^2) - L g
    \]
    finds no stationary points in the interior of $[-1,1]^3$. There are no other meaningful boundary
    conditions to impose.
\end{List}
We are thus left with the research of stationary points for $L_{5,2}$ in the open set $(-1,1)^4$. We have
four quantities to consider, which are
\[
\frac{\partial L}{\partial x_1},
\quad
\frac{\partial L}{\partial x_2},
\quad
\frac{\partial L}{\partial x_3},
\quad
\frac{\partial L}{\partial g} (1-g^2) - L g.
\]
We factorize each one of these quantities and we keep only the non-trivial factors (i.e. factors which do
not have roots on the boundary or for some variable equal to zero): we call these factors $L_{x_1}$,
$L_{x_2}$, $L_{x_3}$ and $L_g$ and we study the system $L_{x_1}=L_{x_2}=L_{x_3}=L_g=0$. Then, as in
previous cases, we consider their resultants in order to look for common zeros: we compute then
\[
\Res(L_{x_2},L_{x_1};g),
\quad
\Res(L_{x_3},L_{x_1};g),
\quad
\Res(L_g,L_{x_1};g),
\]
we factorize them and we keep the non-trivial factors, denoting them as $\res1g$, $\res2g$ and $\res3g$.
We study the system $\res1g=\res2g=\res3g=0$ and we compute then the resultants
\[
\Res(\res2g,\res1g;x_3),
\quad
\Res(\res3g,\res1g;x_3).
\]
Now, a particular phenomenon occurs: if we factorize both resultants, it results that they share a common
factor of the form $x_1x_2^2 - 1/4x_2^2 - 3/4$ which may have roots in the interior. We denote this
common factor as $\res C$. Together with this, each one of these two resultants possesses a specific
non-trivial factor which is not shared by the other resultant: we denote these factors as $\res1x_3$ and
$\res2x_3$. We must then proceed considering two distinct subcases.
\medskip\\
CASE 1: We focus on $\res C=0$. This means that we have to extract common roots of the four derivatives
above from the curve $x_1x_2^2 - \frac{1}{4}x_2^2 - \frac{3}{4} = 0$. Luckily, we can recover a relation
between $x_1$ and $x_2$ from this equation of the form
\begin{equation}\label{Relazione_x1}
x_1 = \frac{3+x_2^2}{4x_2^2}.
\end{equation}
Substitution of~\eqref{Relazione_x1} in $\res1g$ and $\res2g$ allows to eliminate the variable $x_1$ and
obtain two new expressions (after factorization) $\res1gS$ and $\res2gS$, which are the only non-trivial
factors arising from this substitution. We can then compute the resultant
\[
\Res(\res1gS,\res2gS;x_3)
\]
whose only non-trivial factor has the form $x_2^8 - 4x_2^6 - \frac{17}{16}x_2^4 - \frac{45}{8}x_2^2 +
\frac{27}{16}$. We obtain then the possible stationary points of this case by looking at the system
\[
\begin{cases}
& x_2^8 - 4x_2^6 - \frac{17}{16}x_2^4 - \frac{45}{8}x_2^2 + \frac{27}{16} = 0 \\
& \res1gS = 0\\
& \text{resC} = 0\\
& L_{x_1} = 0.
\end{cases}
\]
Studying the first equation, one verifies that substitution of the roots of this equation in the second
one give values for $x_1$ which are bigger than 1 in absolute value: hence Case 1 does not yield any
stationary point in our domain.
\medskip\\
CASE 2: We proceed with the usual resultant tree by computing
\[
\Res(\res1x_3,\res1x_3;x_2).
\]
This gives several non-trivial factors which we can multiply with each other to obtain a unique
polynomial $\res x_2$ in the variable $x_1$. We can then proceed backwards as before and study the
system
\[
\begin{cases}
& \res x_2 = 0\\
& \res1x_3 = 0\\
& \res1g   = 0\\
& L_{x_1}  = 0.
\end{cases}
\]
The first equation has only $x_1$ as a variable, the second one has $x_1$ and $x_2$, and so on. This
makes the discussion of the system easier and it allows to show that it has no solutions. Hence, the
function $L_{5,2}$ does not have stationary points in the open set $(-1,1)^4$, and from the previous
discussion we can conclude that the maximum of $L_{5,2}$ over $[-1,1]^4$ is $16\, M$.

\section{Graphical schemes for the proof of Theorem~\ref{theoremDegreesHigher}}\label{sectionSchemesForTheorem2}
Unfortunately the resultant tree procedure employed in the previous section cannot be applied to the
optimization of $P_{n,1}$ when $n \geq 6$. This happens because the increased degree of $P_{n,1}$ not
only makes more complicated the factorization of the resultants in the tree, but also because those
resultants share more and more factors that should be addressed with ad hoc arguments that however are
possible only when their degree is small, and this is no more the case: the argument we have used to
handle Case~2 described before is unavailable in general. Thus, we renounce to detect the exact maximum
and we settle for computing an upper bound by splitting $P_{n,1}$ into several blocks and producing upper
bounds for each block. However, the splitting is strongly influenced by the signs of the variables and in
order to improve the final result we have to treat separately the $2^{n-1}$ subcases corresponding to a
well determined choice for the signs of each variable.
%
%
The next result shows that we can restrict the range for $g$ to $[0,1]$: this halves the total number of
cases.
\begin{lemma}
Let $L_{n,k}(x_1,\ldots,x_{n-2},g)$ be the $k$-th ordering of $P_{n,1}$. Then the maximum of $L_{n,k}$
over $[-1,1]^{n-1}$ is assumed at $g\geq 0$.
\end{lemma}
\begin{proof}
Remember that the $k$-th ordering is defined by the couple of complex conjugated numbers
$(y_{n-k},\overline{y_{n-k}})$, and that we can assume $1\leq k\leq \intpart{(n-1)/2} +1$. If $k=1$, then
in the transformation given by the change of variables~\eqref{changeOfVarComplex} the variable $g$
appears always paired with $x_{n-2}$ in the form $x_{n-2}\cdot g$. This proves that
\[
L_{n,1}(x_1,\ldots,x_{n-3},-x_{n-2},-g) = L_{n,1}(x_1,\ldots,x_{n-3},x_{n-2},g).
\]
For $2\leq k\leq \intpart{(n-1)/2} +1$, the variable $g$ is always multiplied with either $x_{n-k}$
or $x_{n-k-1}$; moreover, a product containing $x_{n-k-1}x_{n-k}g$ never appears. Hence, we have
\[
L_{n,k}(x_1,\ldots,-x_{n-k-1},-x_{n-k},\ldots,-g) = L_{n,k}(x_1,\ldots,x_{n-k-1},x_{n-k},\ldots,g).
\]
Both these symmetries show that we can reduce to the case $g\geq 0$ for the search of the maximum.
\end{proof}
Starting from this assumption, we now set the construction of graphical schemes as for the totally real
case: we choose a vector of signs $\epsi=(\epsi_1,\ldots,\epsi_{n-2})$ and instead of studying the
function $L_{n,k}$ we study the analogous function with $\prod_{k=i}^j x_k$ replaced by $\prod_{k=i}^j
\epsi_k\prod_{k=i}^j x_k$, so that we have to study $2^{n-2}$ functions defined over $[0,1]^{n-1}$. Just
like for the real case, we denote these functions as \emph{configurations}. The main difference with
the totally real case is that now we have to take into account also the terms of the ordering in which
the variable $g$ appears, together with the square root term: the factors containing $g$ will be labeled
with a mark `` ' '' after the sign $+$ or $-$, and the square root will be denoted as $R^{1/2}$. Thus,
the graphical schemes of the configurations of $L_{6,1}$, $L_{6,2}$ and $L_{6,3}$ with the vector of
signs $(+,-,-,+)$ are
\[
\ytableausetup{nosmalltableaux}
\begin{ytableau}
 {+}   & {-}   & {+}   & *(\CoTa){+'}   & \none\\
 \none & {-}   & {+}   & *(\CoTa){+'}   & \none\\
 \none & \none & {-}   & *(\CoTa){-'}   & \none\\
 \none & \none & \none & *(\CoTa){+'}   & \none[\qquad 2R^{1/2}]
\end{ytableau}
\qquad \raisebox{1cm}{,} \quad
\ytableausetup{nosmalltableaux}
\begin{ytableau}
 {+}   & {-}   & *(\CoTa){+'}   & {+}            & \none\\
 \none & {-}   & *(\CoTa){+'}   & {+}            & \none\\
 \none & \none & *(\CoTa){-'}   & {-}            & \none\\
 \none & \none & \none          & *(\CoTa){+'}   & \none[\qquad 2R^{1/2}]
\end{ytableau}
\qquad \raisebox{1cm}{,} \quad
\ytableausetup{nosmalltableaux}
\begin{ytableau}
 {+}   & *(\CoTa){-'}   & {+}            & {+}            & \none\\
 \none & *(\CoTa){-'}   & {+}            & {+}            & \none\\
 \none & \none          & *(\CoTa){-'}   & *(\CoTa){-'}   & \none\\
 \none & \none          & \none          & {+}            & \none[\qquad 2R^{1/2}]
\end{ytableau}
\qquad \raisebox{1cm}{.}
\]
The three schemes above represent respectively the functions
\begin{align*}
    & (1-x_1)         (1+x_1 x_2)  (1-x_1 x_2 x_3)    (1-2x_1 x_2 x_3 x_4 g + (x_1 x_2 x_3 x_4)^2)\\
    &\phantom{(1-x_1)}(1+x_2)\ \ \ (1-x_2 x_3)\ \ \   (1-2x_2 x_3 x_4 g + (x_2 x_3 x_4)^2)        \\
    &\phantom{(1-x_1)(1+x_1 x_2)}  (1+x_3)\ \ \ \ \ \ (1+2 x_3 x_4 g + (x_3 x_4)^2)               \\
    &\phantom{(1-x_1)(1+x_1 x_2)(1-x_1 x_2 x_3)}      (1-2 x_4 g +  x_4^2)\cdot2\sqrt{1-g^2},     \\[4.5\baselineskip]
    & (1-x_1)         (1+x_1 x_2)  (1-2 x_1 x_2 x_3 g +(x_1 x_2 x_3)^2)          (1-x_1 x_2 x_3 x_4)\\
    &\phantom{(1-x_1)}(1+x_2)\ \ \ (1-2 x_2 x_3 g +(x_2 x_3)^2)\ \ \ \ \ \       (1-x_2 x_3 x_4)    \\
    &\phantom{(1-x_1)(1+x_1 x_2)}  (1+2 x_3 g + x_3^2)\quad\quad\quad\quad\quad\,(1+x_3 x_4)        \\
    &\phantom{(1-x_1)(1+x_1 x_2)(1-2 x_1 x_2 x_3 g +(x_1 x_2 x_3)^2)}            (1-2 x_4 g +  x_4^2)\cdot2\sqrt{1-g^2},\\[.3\baselineskip]
    & (1-x_1)         (1+2 x_1 x_2 g +(x_1 x_2)^2)        (1-x_1 x_2 x_3)\ \ \    (1-x_1 x_2 x_3 x_4)          \\
    &\phantom{(1-x_1)}(1+2 x_2 g +x_2^2)\quad\quad\ \ \ \,(1-x_2 x_3)\quad\ \ \ \,(1-x_2 x_3 x_4)              \\
    &\phantom{(1-x_1)(1+2 x_1 x_2 g +(x_1 x_2)^2)}        (1+2 x_3 g + x_3^2)     (1+2 x_3 x_4 g + (x_3 x_4)^2)\\
    &\phantom{(1-x_1)(1+2 x_1 x_2 g +(x_1 x_2)^2)(1+2 x_3 g + x_3^2)}(            1-x_4)\cdot2\sqrt{1-g^2}.
\end{align*}
A possible way to estimate this new kind of schemes could be using an algorithm employing dynamical
estimates just like the one for totally real fields, so that all configurations are reduced to a standard
one for which static estimates can be applied. Unfortunately, this algorithm does not seem to be
available due to the new terms with the variable $g$. As an example, we no longer have an estimate of the
form $\ytableausetup{smalltableaux,aligntableaux=bottom}
\begin{ytableau}
 {+} & {-'} \\
\end{ytableau}
\leq
\begin{ytableau}
 {-} & {+'} \\
\end{ytableau}
$. %
In fact, this would be equivalent to the estimate $(1-x)(1+2xyg+(xy)^2) \leq (1+x)(1-2xyg+(xy)^2)$ for
$x,y,g\in [0,1]$, but it easy to verify that, while this is true whenever $g=0$, there are instead values
in the admissible range for which it is false.
Similarly, the estimate
$\ytableausetup{smalltableaux,aligntableaux=bottom}
\begin{ytableau}
 {-} & {+'} \\
 {+} & {-'} \\
\end{ytableau}
\leq
\begin{ytableau}
 {+} & {-'} \\
 {-} & {+'} \\
\end{ytableau}
$ %
is not true. However, some dynamical estimates in this new setting are still possible, as proved in the
following lemma.
\begin{lemma}\label{lemmaDynamicalNew}
We have
$\ytableausetup{smalltableaux,aligntableaux=bottom}
\begin{ytableau}
 {+'} & {-'} \\
\end{ytableau}
\leq
\begin{ytableau}
 {-'} & {+'} \\
\end{ytableau}
$
and
$\ytableausetup{smalltableaux,aligntableaux=bottom}
\begin{ytableau}
 {-'} \\
 {+'} \\
\end{ytableau}
\leq
\begin{ytableau}
 {+'} \\
 {-'} \\
\end{ytableau}
$.
\end{lemma}
\begin{proof}
Both these estimates correspond to
\[
(1-2xg+x^2)(1+2xyg+(xy)^2) \leq (1+2xg+x^2)(1-2xyg+(xy)^2).
\]
To prove this, subtract the left hand side from the right hand side: the result is factorized as
$4xg(1-y)(1-yx^2) \geq 0$ (remember that we are assuming all variables to be in $[0,1]$).
\end{proof}
Together with this dynamical estimate, it is also possible to detect new patterns which involve the terms
with $g$, so that the scheme can be covered with patterns and estimated by multiplying the upper bounds
of every pattern. An example of this static inequality is given by the following.
\begin{lemma}\label{lemmaStaticNew1}
We have
$\ytableausetup{smalltableaux,aligntableaux=bottom}
\begin{ytableau}
 {+} & {-'} \\
\end{ytableau}
\leq
32/27$
and
$\ytableausetup{smalltableaux,aligntableaux=bottom}
\begin{ytableau}
 {-'} \\
 {+}  \\
\end{ytableau}
\leq
32/27
$.
\end{lemma}
\begin{proof}
The claim is equivalent to proving that $(1-x)(1+2xyg+(xy)^2) \leq 32/27$. Now, under the assumption that
the variables are all in $[0,1]$, it is immediate to see that the left hand side is maximized for $y = g
= 1$, and it becomes $(1-x)(1+x)^2$: this quantity is maximized at $x=1/3$, giving the value $32/27$ as
maximum.
\end{proof}
The combination of dynamical and static estimates is the tool which allows to obtain estimates for the
configurations in all orderings and degree between $6$ and $9$ as reported in
Theorem~\ref{theoremDegreesHigher}. Differently from the totally real case, the main difficulty is now
the static part of the approach: our method relies much more on this kind of estimates, since the
presence of the terms with $g$ forbids several dynamical estimates. As a consequence, the number of
patterns that have to be recognized is much larger, around 150, and for several of them the optimization
can be quite troublesome. We gathered all the patterns we recognized, together with the corresponding
estimates and their proofs, in the dataset~\cite{DATA}: in the rest of this section we do not show the
proof of every such inequality, but we discuss some of the most meaningful ones. Additional discussion on
the dataset \cite{DATA} and how we used it for the proof of Theorem~\ref{theoremDegreesHigher} can be
found in the next section.

\subsection{Dynamical estimates}
In some cases, an approach completely similar to the one of Lemma~\ref{lemmaDynamical} can be useful:
some configurations for our new graphical schemes can be proved to be less than other ones by using only
dynamical estimates, which can be either the old ones described in Lemma~\ref{lemmaDynamical} or the new
ones given in Lemma~\ref{lemmaDynamicalNew}. An example is the configuration of $P_{6,1}$, first
ordering, defined by the vector of signs $(+,-,+,+)$:
\[
\ytableausetup{nosmalltableaux}
\begin{ytableau}
 {+}   & {-}   & {-}   & {-'}   & \none\\
 \none & {-}   & {-}   & {-'}   & \none\\
 \none & \none & {+}   & {+'}   & \none\\
 \none & \none & \none & {+'}   & \none[\qquad 2R^{1/2}].
\end{ytableau}
\]
In fact, we can apply the estimate H of Lemma~\ref{lemmaDynamical} to the couple $\{(1,1),(1,2)\}$, the
estimate V of the same lemma to $\{(2,3),(3,3)\}$ and the vertical estimate of
Lemma~\ref{lemmaDynamicalNew} to $\{(1,4),(4,4)\}$. This proves that the considered configuration is
estimated by the one (with degree $6$ and with first ordering) defined by the vector of signs $(-,-,-,-)$.
This new configuration requires an ad hoc estimate with static inequalities (which show that the upper
bound is $32$, see~\cite{DATA}), but we have proved that several configurations in this ordering reduce to
this case. Similar reductions occur in other degrees and orderings, though being a small fraction of all
the possible configurations.

\subsection{Easy static estimates}
Although the presence of the terms with $g$ prevents reducing the discussion to a mostly dynamical
approach, several factors of the new graphical schemes are simple enough to be estimated with
inequalities that do not require complicated optimization: a first example is the inequalities
presented in Lemma~\ref{lemmaStaticNew1}. Other instances of this phenomenon is the inequality described
in the following lemma.
\begin{lemma}\label{lemmaStaticNew2}
We have
$\ytableausetup{smalltableaux,aligntableaux=bottom}
\begin{ytableau}
 {+} & {+} &  {-'} \\
\end{ytableau}
\leq 1
$.
\end{lemma}
\begin{proof}
The claim is equivalent to proving that $(1-x)(1-xy)(1+2xyzg+(xyz)^2) \leq 1$ with all the variables in
$[0,1]$. It is clear that the polynomial is maximized at $z=g=1$, providing
\[
(1-x)(1-xy)\cdot(1+xy)^2 \leq (1-x)(1-xy)\cdot(1+x)(1+xy) = (1-x^2)(1-(xy)^2) \leq 1.
\]
\end{proof}
An application of this can be seen for the estimate of the configuration of $P_{8,1}$, second ordering,
defined by the signs $(+,-,-,+,-,-)$, whose graphical scheme is
\begin{equation}
\ytableausetup{nosmalltableaux}\label{SchemeEasyStatic}
\begin{ytableau}
*(\CoTa){+}   & *(\CoTa){-}   & *(\CoTa){+}   & *(\CoTa){+}   & *(\CoTa){-'}   & *(\CoTa){+}   & \none\\
\none         & *(\CoTa){-}   & *(\CoTa){+}   & *(\CoTa){+}   & *(\CoTa){-'}   & *(\CoTa){+}   & \none\\
\none         & \none         & {-}           & {-}           & {+'}           & {-}           & \none\\
\none         & \none         & \none         & {+}           & {-'}           & {+}           & \none\\
\none         & \none         & \none         & \none         & {-'}           & {+}           & \none\\
\none         & \none         & \none         & \none         & \none          & {-'}          & \none[\qquad 2R^{1/2}].
\end{ytableau}
\end{equation}
Assuming an estimate for the configuration of $P_{6,1}$, second ordering signs $(-,+,-,-)$, has been
found (\cite{DATA} reports that it is $12.33$), one can find a bound for this configuration by simply
estimating the first two lines. In fact, the signs at positions $\{(1,1),(1,2)\}$, $\{(1,6)\}$ and
$\{(2,6)\}$ are patterns described in Lemma~\ref{lemmaStaticEstimates} and such that they are bounded by
$1$. Similarly, the positions $\{(1,3),(1,4),(1,5)\}$ and $\{(2,3),(2,4),(2,5)\}$ form the pattern
described in Lemma~\ref{lemmaStaticNew2} and so they are bounded by 1 too. Finally, the remaining
position $\{(2,2)\}$ is trivially bounded by 2: this means that the first two lines of this scheme are
bounded by $2$ and thus the examined configuration is bounded by $24.66$. This argument shows that the
optimization of this configuration relies critically on the result for the configuration in degree 6, for
which more accurate estimates are needed.

\subsection{Less easy static estimates}
As we said, the estimate of the previous configuration relies heavily on the upper bound for the
graphical scheme
\[
\ytableausetup{nosmalltableaux}
\begin{ytableau}
 {-}   & {-}   & {+'}   & {-}   & \none\\
 \none & {+}   & {-'}   & {+}   & \none\\
 \none & \none & {-'}   & {+}   & \none\\
 \none & \none & \none & {-'}   & \none[\qquad 2R^{1/2}].
\end{ytableau}
\]
This can be decomposed in three blocks, which are $A\coloneqq \{(1,1),(1,4),(2,4)\}$, $B\coloneqq
\{(1,2),(1,3),$ $(2,2),(2,3)\}$ and $C\coloneqq \{(3,3),(3,4),(4,4)\}$ (this one multiplied with
$2R^{1/2}$). While the block $A$ is immediately seen to be bounded by $2$ thanks to
Lemma~\ref{lemmaStaticEstimates}, the remaining two blocks are more complicated and for them we cannot
give a quick proof just like for the inequality of Lemma~\ref{lemmaStaticNew2}. Nonetheless, the
optimization of these patterns can be obtained following the very same resultant tree strategy which we
employed for the proof of Theorem~\ref{theoremDegree5}.
\begin{lemma}\label{lemmaStaticNew3}
We have $\ytableausetup{smalltableaux,aligntableaux=bottom}
\begin{ytableau}
 {-} & {+'} \\
 {+} & {-'} \\
\end{ytableau}
\leq 32/27$
and
$\ytableausetup{smalltableaux,aligntableaux=bottom}
\begin{ytableau}
 \none     & \none[j] & \none[j'] \\
 \none[i]  & {-}      & {+}       \\
 \none[i'] & \none    & {-}       \\
\end{ytableau}
\, 2R^{1/2} \leq 5.2$ for $j'=j+1$ or $i'=i+1$.
\end{lemma}
\begin{proof}
Let us begin with the first pattern. We want to prove that the function
\[
F = (1+xy)(1-2xyzg+(xyz)^2)(1-y)(1+2yzg+(yz)^2)
\]
is at most $32/27$ for $(x,y,z,g)\in [0,1]^4$. Just like for the proof of $L_{5,1}$, we begin the
optimization by studying the behaviour of this function on the boundary of the hypercube, and then we
shall look for eventual stationary points in $(0,1)^4$ (all the factorizations computed in the next lines
are done via MAGMA, while all the computations of real roots of rational polynomials are done in PARI).
We begin by studying what happens whenever one of the variables is 0 (notice that this is indeed a
boundary computation, since differently from Theorem~\ref{theoremDegree5} we are assuming all the
variables to be non-negative); later we shall study the case of the variables being equal to $1$.
\begin{itemize}
    \item[$x=0$:] the function becomes $(1-y)(1+2yzg+(yz)^2)$, which is the one described in
        Lemma~\ref{lemmaStaticNew1}: hence we already know its maximum is $32/27$.
    \item[$y=0$:] the function trivially becomes identical to 1.
    \item[$z=0$:] the function becomes $(1+xy)(1-y) \leq (1-y^2) \leq 1$.
    \item[$g=0$:] the function is now $(1+xy)(1+(xyz)^2)(1-y)(1+(yz)^2)$: this is maximized at $x=z=1$ giving
    \[
    (1+y)(1+y^2)(1-y)(1+y^2) = (1-y^4)(1+y^2) \leq 32/27
    \]
    (this follows again by Lemma~\ref{lemmaStaticNew1} considering $y^2$ as the variable).
   \item[$x=1$:] the function has now the form
   \[
   (1+y)(1-2yzg+(yz)^2)(1-y)(1+2yzg+(yz)^2) = (1-y^2)[(1+(yz)^2)^2-4(yzg)^2].
   \]
   Notice that, under our hypotheses, this expression is maximized at $g=0$: thus we reduce to the
   previous boundary computation, so that we already know that the function is at most $32/27$ in this
   case too.
   \item[$y=1$:] the function is trivially equal to 0.
   \item[$z=1$:] the function becomes $ L =(1+xy)(1-2xyg+(xy)^2)(1-y)(1+2yg+y^2)$: this does not have a
       direct estimate and so we proceed with the resultant tree, i.e. we look for stationary points of
       $L$ in $(0,1)^3$, i.e. we want to solve
       \[
       \frac{\partial L}{\partial x} = \frac{\partial L}{\partial y} = \frac{\partial L}{\partial g} = 0.
       \]
       We factor these derivatives and denote by $L_x$, $L_y$ and $L_g$ the only factors of each
       derivative which are not trivial (in the sense that they have not only roots on the boundary or
       they are not always positive over the domain). Then we compute $\Res(L_x,L_g;g)$ and
       $\Res(L_y,L_g;g)$, and we call $\res1g$ and $\res2g$ their unique non-trivial factors. We are then
       considering the system $L_x = \res1g = \res2g = 0$ and thus we compute the resultant
       $\Res(\res1g,\res2g;x)$: however, the polynomial we obtain has only factors which are powers of
       $y$ or $(y-1)$, hence trivial: this means that it is not possible to find a stationary point for
       $L$.

       The discussion on $L$ concludes by considering the only boundary condition we can add to $z=1$,
       i.e. $g=1$: in this subcase the function is now $(1+xy)(1-xy)^2(1-y)(1+y)^2$, and since its
       derivative in $x$ has no zeros in $(0,1)$, this case reduces to previous ones and so is not of
       interest.
   \item[$g=1$:]  this case is very similar to the one above, since again we have a function in three
       variables for which a resultant tree is required. Fortunately, this is even easier since, denoting
       by $L_y$ and $L_z$ the unique non-trivial factors of $\partial L/\partial y$ and $\partial
       L/\partial z$ are such that $\Res(L_y,L_z;z)$ has only trivial factors, and since we do not
       have additional boundary conditions to impose, we can skip this case.
  \end{itemize}
Finally, we consider the study over the interior $(0,1)^4$. This requires a resultant tree starting from
four derivatives, so a priori it appears to be more complicated: however, if $F_z$ and $F_g$ are the
unique non-trivial factors of the derivatives of $F$ with respect to $z$ and $g$, one verifies that
$\Res(F_z,F_g;g)$ has only trivial factors, so that the two derivatives have no common roots in the
interior and thus stationary points for $F$ do not exist. This concludes the optimization of the first
pattern.

The optimization of the second pattern is completely similar, since it is equivalent to proving that
\[
S =   \underbrace{(1+2xg+x^2)(1-xy)(1+2yg+y^2)\cdot 2}_F\cdot\sqrt{1-g^2}
\]
is bounded by $5.2$ and this is done via a simpler resultant tree that only involves three variables. The
only thing we remark is that the derivatives are made for the function $F$ and that instead of the one
with respect to $g$ we consider the quantity $\big(\frac{\partial F}{\partial g} (1-g^2)- F g\big)$. In
this way, the resultant tree procedure shows that the function $S$ has no stationary points in $(0,1)^3$
and its maximum is assumed at the boundary values $(x,y,z)=(0,1,1/2)$ and is equal to $3\sqrt{3} < 5.2$.
\end{proof}
The results of Lemma~\ref{lemmaStaticNew3} allow to estimate the blocks $B$ and $C$ of the scheme we are
considering: we multiply then the upper bounds of every block and we end with the upper bound $2\cdot
32/27 \cdot 5.2 = 12.3259\ldots < 12.33$ for this configuration.

Many other patterns we found are estimated with a resultant tree procedure in three or four variables
like the one described above. However, not all the necessary patterns involve such a small number of
variables or an easy procedure.

\subsection{Complicated static estimates}
The $2^4$ different configurations of $L_{6,1}$ (i.e. the number of transformations of $P_{6,1}$ in the
first ordering we get with the different choice of the signs of their variables) can be estimated using
static or dynamic inequalities which are in similar shape to the ones we have briefly described before,
or in the worst case not so much more complicated. This is no more the case whenever we change the
ordering and/or the degree: though many of the configurations can still be bounded in similar fashion,
some other ones do need of a more detailed investigation, especially in order to obtain an upper bound
which is suitable for our original number-theoretic goal. We illustrate this with the following example,
which is the configuration of $L_{8,4}$ (i.e. fourth ordering in degree 8) with the vector of signs
$(+,-,-,-,-,-)$.
\[
\ytableausetup{nosmalltableaux}
\begin{ytableau}
{+}   & {-}   & {+'}   & {-}   & {+ }   & {-}   & \none\\
\none & {-}   & {+'}   & {-}   & {+ }   & {-}   & \none\\
\none & \none & {-'}   & {+}   & {- }   & {+}   & \none\\
\none & \none & \none  & {-'}  & {+'}   & {-'}  & \none\\
\none & \none & \none  & \none & {-}    & {+}   & \none\\
\none & \none & \none  & \none & \none  & {-}   & \none[\qquad 2R^{1/2}].
\end{ytableau}
\]
This is a complicated scheme to estimate, due to the presence of multiple factors to consider and the
symmetric disposition of the $g$-terms. In fact, none of the estimates of the previous kind is able to
give a satisfying estimate for this configuration (the concept of ``satisfying'' will be further explained
in the next sections). It was thus necessary to provide more complicated decompositions of this scheme,
each one requiring a much longer optimization process (similar to the one employed for the estimate of
$L_{5,2}$) and the one we gave is made of the following two blocks: the first one is the gray one which
also includes the factor $2R^{1/2}$, while the second one is formed by the remaining of the scheme.
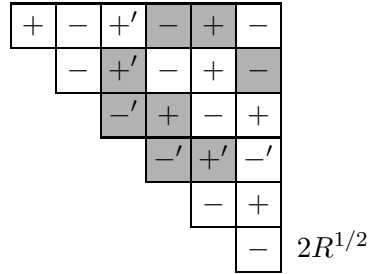
\[
\ytableausetup{nosmalltableaux}
\begin{ytableau}
{+}   & {-}   & {+'}          & *(\CoTa){-}   & *(\CoTa){+ }   & {-}         & \none\\
\none & {-}   & *(\CoTa){+'}  & {-}           & {+ }           & *(\CoTa){-} & \none\\
\none & \none & *(\CoTa){-'}  & *(\CoTa){+}   & {- }           & {+}         & \none\\
\none & \none & \none         & *(\CoTa){-'}  & *(\CoTa){+'}   & {-'}        & \none\\
\none & \none & \none         & \none         & {-}            & {+}         & \none\\
\none & \none & \none         & \none         & \none          & {-}         & \none[\qquad 2R^{1/2}].
\end{ytableau}
\]
\begin{lemma}\label{lemmaStaticNew4}
The gray block multiplied with $2R^{1/2}$ is bounded by $9.482$.
\end{lemma}
We do not give here the complete proof of this lemma and we refer instead our database~\cite{DATA} which
contains the complete proof written in MAGMA and PARIgp files; nonetheless, we show one of the main
difficulties occurred in the estimate of this block. In fact, the function we need to estimate has the
form
\begin{align*}
 S = (1+xyzt)(1-xyzta)&(1-2yzg+(yz)^2)(1+yztab)(1+2zg+z^2)(1-zt)\\
                      &(1+2tg+t^2)(1-2tag+(ta)^2)\cdot 2\sqrt{1-g^2}.
\end{align*}
We now have six variables to consider, and this results in considering a system of six polynomial
equations to solve. Again, we start by the derivatives of $S$ with respect to every variable (with the
derivative with respect to $g$ given in terms of $\partial L/\partial g\cdot (1-g^2) - L\cdot g$, where
$L\coloneqq S/(2\sqrt{1-g^2})$). Factorization gives some non trivial factors that we gather into a
system
\[
L_x = L_y = L_z = L_t = L_a = L_g =0
\]
and the usual resultant tree and factorization process (done with resultants of these factors with each
other with respect to the variable $a$, $g$ and then $x$) gives
\[
\res1a = \res2a = \res3a = \res4a = \res5a = 0,
\]
\[
\res1g = \res2g = \res3g = \res4g = 0,
\]
\[
\res1x = \res2x = \res3x = 0.
\]
The next step consists in looking for the common zeros of the last three resultants, and so we compute
and factor the resultants $\Res(\res2x,\res1x;y)$ and $\Res(\res3x,\res1x;y)$. However, now there is an
abundance of non-trivial factors, and most of them are common between the two resultants: in fact,
$\Res(\res2x,\res1x;y)$ and $\Res(\res3x,\res1x;y)$ share five common non-trivial factors $\res C1$,
$\res C2$, $\res C3$, $\res C4$, $\res C5$ and they both posses a specific and different non-trivial
factor, which are respectively $\res1y$ and $\res2y$. This forces to consider six subcases, even if five
of them are very similar.
\medskip\\
CASES 1-5: we consider the common factor $\res Ci$ (with $i\in\{1,\ldots,5\}$) and we look for
stationary points of $S$ over the locus $\res Ci=0$. These common factors are polynomial expressions
which are simpler than the ones describing the factors of the derivatives, and thus one can consider the
system
\[
L_x = L_y = L_z = L_t = L_a = \res Ci =0
\]
and restart the resultant tree process, which is now easier and with simpler resultants since a
complicated condition has been replaced with a simpler one. Iteration of this process eventually leads to
the desired polynomial in one variable from which we can extract roots and move backwards to search for
stationary points.

In these specific cases, the factors $\res Ci$ are always linear in some variable, and thus the
relation $\res Ci=0$ allows to substitute one variable in a chosen previously computed polynomial:
in particular, if we substitute in either $\res1x$ or $\res1g$, we obtain polynomials which
have no roots in the open interval $(0,1)$.
\medskip\\
CASE 6: we proceed with the resultant tree computing $\Res(\res1y,\res2y;z)$: we multiply the non-trivial
factors of this resultant obtaining a polynomial $\res1z$ and finally we solve the system
\[
L_y = \res1a = \res1g = \res1x = \res1y = \res1z = 0
\]
where the first polynomial has six variables, the second five, the third four and so on. The search of
the roots and the evaluation give some stationary points, which however give small local maximums (around
$1.41\ldots$) that are easily dominated by values on the boundary.

The detailed optimization of the interior and the boundary is reported in~\cite{DATA} and similar
difficulties (sometimes with worse expressions) are encountered in the estimate of the second block and
in the optimization of some other configurations for which the old procedures do not work suitably. They
all share, however, the same ``resultant tree/common factors'' behaviour that we described with this
example.

A similar optimization procedure shows that the non-coloured part of the scheme, if considered as a
unique block, is estimated by $8.641$. Thus the scheme is bounded by the product of the upper bounds of
the two blocks, which is $82$. We further improve this result verifying that the gray block is
bounded by $8$ when $g\in [1/2,1]$ and that the non-coloured part is bounded by $8.25$ when $g\in[0,1/2]$.
Therefore the scheme is actually bounded by $\max(9.482\cdot 8.25, 8\cdot 8.641) \leq 78.3$.

\section{Technical and computational remarks}\label{sectionTechnicalRemarks}
In this section we gather several remarks about our computations and the estimates we needed in order to
deal with all the graphical schemes we considered.
\medskip\\
1) As we mentioned, the optimization in the degree $5$ case, i.e. $P_{5,1}$, was made considering each
variable lying in $[-1,1]$: the function was simple enough in both orderings to provide a rigorous
optimization. For $P_{n,1}$ with $n\in\{6,7,8,9\}$, only variables between 0 and 1 were considered.
\medskip\\
2) We have three (respectively three, four, four) orderings for $P_{6,1}$ (respectively $P_{7,1}$,
$P_{8,1}$, $P_{9,1}$). Each ordering unfolds into 16 configurations (respectively 32, 64, 128), each one
defined by a vector of signs. Every configuration is estimated by recognizing patterns into it and either
bounding them via static estimates or replacing them with other patterns via dynamic estimates. In
several cases this is achieved just estimating the first line or the last column of the scheme and then
multiplying the obtained upper bound with the one previously found for the remainder of the scheme, which
results to be a configuration in smaller dimension which has been studied before. The scheme shown
in~\eqref{SchemeEasyStatic} that we have already discussed earlier is an example of this procedure.
\medskip\\
3) The dataset~\cite{DATA} contains the collection of all static and dynamic estimates we employed for
the optimization of the orderings: more in detail, they are gathered in the file
``estimates-general-database.txt''. The file contains slightly more than 150 estimates, divided into 9
groups. Group A) contains the estimates first employed for the totally real case and described in Lemma
\ref{lemmaDynamical} and Lemma~\ref{lemmaStaticEstimates}. Groups from B) to G) collect estimates
involving the terms containing the new variable $g$: their labelling follows no particular order, apart
from chronological appearance and similarities in their resolution. Moreover, these estimates are used
mostly for the cases in degree $6$, $7$ and $8$. Group H) contains inequalities which turned useful for
dealing with the degree 9 case: some of them were later used in lower degree cases replacing some
estimates of the previous groups since they turned to be easier to apply. Group I) contains unused but
proved estimates. Finally, the tiny group RMK contains four inequalities which allow to reduce the number
of variables in the optimization of a specific scheme and to estimate a large amount of factors in a
similar way to the ``triangle estimate'' of Lemma~\ref{lemmaStaticEstimates}.
\medskip\\
4) Apart from the very easy ones, the estimates in this collection are all proved by eliminating the
variables via the iterated computation of the resultants, as explained in the previous section: the
factorization process was carried most of the times in MAGMA (which is very suitable for the
factorization of rational polynomials in several variables) while the computation of real roots and
evaluation of polynomials was done mostly in PARIgp.\\
For the search of the solutions of these systems we have also considered the possibility to take
advantage of more advanced tools in Elimination Theory, as Gr\"{o}bner bases, and using MAGMA and {\tt
msolve}~\cite{BerthomieuEderSafey-El-Din} for these computations. However, apart from some marginal
success, we have discovered that both softwares are not strong enough to simplify in some significant way
our arguments. For example, MAGMA cannot compute a Gr\"{o}bner basis for the system giving the stationary
points of $L_{6,1}$. On the other hand, {\tt msolve} computes the basis very quickly, but then cannot
compute the solutions since the system contains a positive-dimensional variety that {\tt msolve} is not
able to make explicit. This is unfortunate, since a direct attack of this system would be able to compute
the true maximum for $P_{6,1}$. The same phenomenon appears with the systems which come with many other
inequalities we considered for the proof of Theorem~\ref{theoremDegreesHigher} (see Cases~1-5 which are
described in Section~6.4). For this reason we have decided to relay our computations on the old technique
of eliminating variables via an iterated chain of resultants.
\\
The dataset~\cite{DATA} contains a MAGMA file and a PARIgp file for every proved estimate: the PARIgp
part especially contains a program which allows to compute all the roots of the examined polynomials.
\medskip\\
5) The proof of some estimates (especially the ones in group G)) turned out to be longer, since they
involved polynomials with six or seven variables for which many subcases had to be considered during the
optimization. Many of these cases required intermediate considerations due to the presence of the common
factors between the considered resultants: they were dealt with in the same way as done in the proof of
the second ordering of Theorem~\ref{theoremDegree5}.
\medskip\\
6) Many configurations are estimated by dividing them in blocks and by providing an estimate for every
block via resultant trees. Sometimes, choosing the blocks for the decomposition is immediate; however,
there are cases where a convenient division in blocks was not evident (such was, for example, the
configuration of the fourth ordering $L_{8,4}$ in degree $8$ given by the vector of signs $(+,-,-,-,-,-)$).
We looked for a good decomposition by writing the object function in MATLAB and by applying the Global
Optimization Toolbox, which provided several groups of blocks and their (numerical) maximums whenever the
variables are between 0 and 1. The MATLAB results, however, were only considered as suggestions and not
as proved facts: whenever we found what seemed a convenient decomposition, we have produced a full proof
that the proposed maximum of each block was indeed correct (the proof was carried again via the resultant
tree procedure in MAGMA/PARIgp).
\medskip\\
7) Sometimes, in order to improve the results and obtain upper bounds suitable for our number-theoretic
goals, a simple division in blocks was not sufficient. In some specific cases, like the configuration of
the third ordering $L_{8,3}$ in degree $8$ given by the vector of signs $(+,-,+,+,-,$ $+)$, we not only
divided the function in two blocks, but we also considered for every block two precise subcases: the
first one is given by assuming a specific variable $x_i$ in an interval $[0,\alpha]$ (where $\alpha\in
(0,1)$ is a convenient number), the second one by assuming $x_i\in [\alpha,1]$. This was done in order to
exploit a cancellation behaviour, since the two blocks do not assume their maximum values on the same
points: in particular, one block will assume its maximum for $x_i\in [0,\alpha]$ and be instead quite
small in $[\alpha,1]$, while the other block will present the opposite behaviour. The necessary
inequalities, with their MAGMA/PARIgp files, are present in the dataset.

\section{Final remarks on the results}\label{sectionFinalRemarks}
The following table represents the upper bounds we detected for each ordering.
\begin{table}[H]\small
    \centering
    \begin{tabular}{|l|c|c|c|c|c|c|c}
    \hline
        \backslashbox{ordering}{degree $n$ and $r_2=1$} & 5 & 6 & 7 & 8 & 9\\
         \hline
        1st & 16M & 32    & 32M   & 64M   & 155.1 \\
        \hline
        2nd & 16M & 32M   & 54M   & 79.42 & 190.2 \\
        \hline
        3rd &     & 34.89 & 65.81 & 79.2  & 201.4 \\
        \hline
        4th &     &       &       & 80    & 233.1 \\
        \hline
    \end{tabular}
    \label{TableMaximum}
\end{table}
First of all, we compare these upper bounds with the basic bound $n^{n/2}$: the new bounds are
considerable improvements being reduced by the factor $3.3$ ($6.2$, $13.8$, $51.2$, $84.4$, respectively)
in the case of degree $5$ ($6$, $7$, $8$, $9$, respectively).

We also know from Theorem~\ref{theoremDegree5} that the new upper bound $16M$ in degree $5$ is the best
possible for each ordering, since this value is attained in both cases at the point $(x_1,x_2,x_3,g) =
(1/\sqrt{7},-1,1,1/(2\sqrt{7}))$. The upper bounds for degree $6$ first ordering and degree $7$ first
ordering are sharp too, since
\[
L_{6,1}((-1,0,-1,1,0)) = 32
\]
and
\[
L_{7,1}((-1,0,1/\sqrt{7},-1,1,1/(2\sqrt{7}))) = 32M.
\]
Notice in particular that the correct bound for degree $6$ first ordering is two times the bound for the
function $P_{4,1}$, and the one in degree $7$ first ordering is twice the bound for $P_{5,1}$. We
conjecture the following:
\begin{List}
 \item For a fixed degree $n$, the maximum of $L_{n,k}$ should be the same independently from the chosen
     $k$-th ordering, exactly like it happens for $n\leq 5$.
 \item the maximum of $P_{n,1}$ follow a recursive formula starting from degree 4 and 5:
     \[
      \sup P_{n+2,1} = 2\sup P_{n,1}
      \quad\text{for $n\geq 4$, so that}\quad
      \sup P_{n,1} =
      \begin{cases}
       2^{(n+4)/2}     & n\geq 4 \text{ even}\\
       2^{(n+3)/2}\, M & n\geq 5 \text{ odd}.
      \end{cases}
     \]
\end{List}
We conclude by considering the degree $9$ case. At the moment the known list of ten fields with signature
$(7,1)$ in LMFDB is not proved to be complete, but one conjectures that this is the case. Assume that
this is true. The PARIgp computation of the regulators yields their true value since the conditions of
Proposition~\ref{prop4} are satisfied, and the smallest value is approximatively $18.0874$ and comes from
the field with smallest discriminant in the list, i.e. the field defined by the polynomial
\begin{equation}\label{polMinDisc(7,1)}
x^9 - x^8 - 5x^7 + 6x^6 + 5x^5 - 11x^4 + 5x^3 + 6x^2 - 6x + 1
\end{equation}
and absolute discriminant equal to $1904081383$. Adjusting the level $R_0$ in Proposition~\ref{prop2} to
$18.1$ and using the upper bound $233.1$ in Theorem~\ref{theoremDegreesHigher} for $P_{9,1}$ we get the
bound $\log |d_K|\leq 47.334357\ldots$. Thus, in order to explore all $(7,1)$ fields having a regulator
$\leq 18.1$ we should explore all fields with absolute discriminant up to $\exp(47.334357\ldots)$.
Unfortunately in this case Proposition~\ref{prop3} does not help to reduce the range, since $4
g_{7,1}(\exp(-47.334357\ldots))$ is negative, and the task remains impossible. The only thing we can
conclude (using the classification method with this upper bound) is that, imposing $R_0:=5.3$,
Proposition~\ref{prop1} gives the upper bound $|d_K|\leq \exp(41.471548\ldots)$: we have also
$2g_{7,1}(\exp(-41.471548\ldots)) = 32.365375 > 5.3$ and $4g_{7,1}(1/1904081382)=16.786962\ldots > 5.3$,
therefore any field of signature $(7,1)$ with $|d_K| > 1904081382$ has regulator above $5.3$, and if
$1904081383$ is indeed the minimal discriminant in this signature, this would mean that every field with
signature $(7,1)$ would have regulator bigger than $5.2$.

The situation does not improve even if we assume for $P_{9,1}$ the conjectured optimal upper bound equal
to $64\, M = 66.786126\ldots$. This is because looking for fields with $R_K\leq 18.1$ is still unfeasible
by the classification method since Proposition~\ref{prop1} now gives $|d_K| \leq \exp(44.834413\ldots)$
and $4 g_{7,1}(\exp(-44.834413\ldots))$ is again negative. The classification remains then not
possible and the only thing we can conclude, by operating like above and assuming the conjectural bound
$64\,M$, is that all fields in signature $(7,1)$ with $|d_K| > 1904081382$ have regulator which is at
least $9.2$.

It appears then that, at least for the fields with one complex embedding, the method for the
classification of number fields with small regulator has reached its limits and new theoretical
improvements are needed for the resolution of the problem: in particular, it would be of interest to know
if more efficient versions of the lower bound~\eqref{lowerBoundRegulator} for $R_K$ can be obtained,
either by improving the existent result using explicit formulae of Dedekind Zeta functions or by new
tools.

\end{document}